\documentclass[a4paper,twoside,10pt,leqno]{amsart}
\usepackage{fullpage}
\usepackage{amsmath} \usepackage{amssymb} \usepackage{amsopn}
\usepackage{amsthm} \usepackage{amsfonts} \usepackage{mathrsfs}
\usepackage{stmaryrd} 
\usepackage[OT2,T1]{fontenc}

\usepackage{verbatim}

\usepackage[pdfpagelabels,pdftex]{hyperref}
\usepackage[usenames]{color}
\usepackage[all]{xy}

\usepackage{setspace}
\usepackage{amssymb}
\usepackage{euscript}
\usepackage{cite}
\usepackage[all]{xy}
\usepackage{tabularx}

\theoremstyle{plain}
\usepackage{amsfonts}
\usepackage{graphicx}
\usepackage{amsmath}

\hypersetup{
  pdftitle={},
  pdfauthor={},
  pdfsubject={},
  pdfkeywords={},
  colorlinks=false,
  breaklinks=true,
  bookmarksopen=true,
  bookmarksnumbered=true,
  pdfpagemode=UseOutlines,
  plainpages=false}

\newcommand{\bC}{{\mathbb C}}

\newcommand{\bN}{{\mathbb N}}

\newcommand{\bQ}{{\mathbb Q}}
\newcommand{\bR}{{\mathbb R}}

\newcommand{\bZ}{{\mathbb Z}}

\newcommand{\cF}{{\mathscr F}}

\newcommand{\cK}{{\mathscr K}}
\newcommand{\cL}{{\mathscr L}}

\newcommand{\caO}{{\mathcal O}}

\newcommand{\fP}{{\mathfrak P}}

\newcommand{\fp}{{\mathfrak p}}
\newcommand{\fq}{{\mathfrak q}}

\DeclareMathOperator{\Hom}{Hom}

\DeclareMathOperator{\Aut}{Aut}

\DeclareMathOperator{\Spec}{Spec}

\DeclareMathOperator{\BS}{BS}

\DeclareMathOperator{\Gal}{G}

\newcommand{\ab}{{\rm ab}}

\newcommand{\Syl}{{\rm Syl}}

\DeclareMathOperator{\Regul}{Reg}

\makeatletter
\newtheorem*{rep@theorem}{\rep@title}
\newcommand{\newreptheorem}[2]{%
\newenvironment{rep#1}[1]{%
 \def\rep@title{#2 \ref{##1}}%
 \begin{rep@theorem}}%
 {\end{rep@theorem}}}
\makeatother

\newtheorem{thm}{Theorem}[section]

\newtheorem{prop}[thm]{Proposition}
\newtheorem{Prop}[thm]{Proposition}

\newtheorem{cor}[thm]{Corollary}

\newtheorem{lm}[thm]{Lemma}

\newreptheorem{thm}{Theorem}
\newreptheorem{prop}{Proposition}
\newreptheorem{cor}{Corollary}

\theoremstyle{definition}
\newtheorem{Def}[thm]{Definition}

\newtheorem{rem}[thm]{Remark}
\newtheorem{rems}[thm]{Remarks}

\newenvironment{pro*}[1][Proof]{{\it{#1:}} }{}

\newcommand\rar{ \rightarrow }
\newcommand\tar{ \twoheadrightarrow }
\newcommand\har{ \hookrightarrow }
\newcommand\Rar{ \Rightarrow }
\newcommand\LRar{ \Leftrightarrow }

\newcommand\dirlim{\mathop{\underrightarrow{\lim} }}
\newcommand\prolim{\mathop{\underleftarrow{\lim} }}
\newcommand{\sm}{{\,\smallsetminus\,}}
\newcommand\N{{\rm N}}
\newcommand\Cl{\rm Cl}

\newcommand\rank{\rm rk}

\newcommand\tame{\mathop{ \rm tm}}
\newcommand\bap{\bar{\fp}}
\newcommand\baq{\bar{\fq}}

\newcommand\Ind{\mathop{ \rm Ind}}

\DeclareMathOperator{\vol}{vol}

\DeclareMathOperator{\rel}{rel}

\newcounter{absatzcounter}[section]
\numberwithin{equation}{section}

\begin{document}

\title[Reconstruction of decomposition subgroups]{Reconstructing decomposition subgroups in arithmetic fundamental groups using regulators}
\author{A. Ivanov}
\email{ivanov@ma.tum.de}
\address{Zentrum Mathematik - M11, Technische Universit\"at M\"unchen, Boltzmannstr. 3, 85748 Garching bei M\"unchen, Germany}
% \classification{14G32 (primary), 11R42, 11R29 (secondary)}
% \thanks{The author was supported by the Technical University Munich}
\begin{abstract}
Our main goal in the present article is to explain how one can reconstruct the decomposition subgroups and norms of points on an arithmetic curve inside its fundamental group if the following data are given: the fundamental group, a part of the cyclotomic character and the family of the regulators of the fields corresponding to the generic points of all \'{e}tale covers of the given curve. The approach is inspired by that of Tamagawa for curves over finite fields but uses Tsfasman-Vl\u{a}du\c{t} theorem instead of Lefschetz trace formula. To the authors' knowledge, this is a new technique in the anabelian geometry of arithmetic curves. It is conditional and depends on still unknown properties of arithmetic fundamental groups. We also give a new approach via Iwasawa theory to the local correspondence at the boundary.

\end{abstract}

\maketitle

\section{Introduction}

Let $K$ be a number field, $S$ a finite set of primes of $K$ containing all archimedean primes and $K_S/K$ the maximal extension of $K$ unramified outside $S$. Let $\caO_{K,S}$ be the ring of $S$-integers of $K$. Then the \'{e}tale fundamental group of $\Spec\caO_{K,S}$ is equal to the Galois group $\Gal_{K,S}$ of $K_S/K$. We discuss in this article, how one can deduce anabelian information on $\Spec\caO_{K,S}$ from $\Gal_{K,S}$ and (essentially) the family of residues at $s = 0$ of zeta functions of all intermediate fields $K_S/L/K$. Our goal, is to establish a \emph{local correspondence on the curve} $\Spec \caO_{K,S}$: i.e., given pairs $(K_i,S_i)$ for $i=1,2$ and an isomorphism $\sigma \colon \Gal_{K_1,S_1} \stackrel{\sim}{\rar} \Gal_{K_2,S_2}$, preserving the product 'class number times regulator' for all finite subextensions, let for any finite subextension $K_{1,S_1}/L_1/K_1$, denote the corresponding extension of $K_2$ via $\sigma$ by $L_2$. For a scheme $X$, let $|X|$ denote the set of closed points of $X$. Under (unknown) additional requirements on $\Gal_{K_i,S_i}$, we will show the existence of a compatible family of bijections

\[ \sigma_{\ast,L_1} \colon |\Spec \caO_{L_1,S_1}|  \stackrel{\sim}{\rar} |\Spec \caO_{L_2,S_2}| \] 

\noindent on closed points, for any finite subextension $K_{1,S_1}/L_1/K_1$, characterized by $\sigma(D_{\bap}) = D_{\sigma_{\ast}(\bap)}$, where $D_{\bap}$ is the decomposition subgroup of the point $\bap$ and $\sigma_{\ast}$ is the inverse limit of the maps $\sigma_{\ast,L_1}$. Moreover, the maps $\sigma_{\ast,L_1}$ preserve absolute norms of primes.

Let us use the following notation: if (x),(y) are some sets of invariants of $K,S$ (like, for example, the position of the decomposition groups of primes inside $\Gal_{K,S}$), then (x) $\rightsquigarrow$ (y) resp. (x) $\leftrightsquigarrow$ (y) will have the following meaning: if the data in (x) are known, then we can deduce the data in (y) from them resp. the knowledge of (x) and (y) is equivalent. In particular, (x) $\rightsquigarrow$ (y) implies that if two pairs $(K_i,S_i), i = 1,2$ are given with $\Gal_{K_1,S_1} \cong \Gal_{K_2,S_2}$ and such that the data in (x) coincide for $i = 1,2$, then also the data in (y) are coincide.

The goal of the first part of the article is to establish $\sigma_{\ast}$ on the primes lying \emph{on the boundary}, i.e., in $S$. Here we enforce the main result from \cite{Iv}, using methods from Iwasawa theory and the Greenberg conjecture (\cite{Gre} Conjecture 3.5). Here is our result in this direction. Let $S_{\infty} = S_{\infty}(K)$ resp. $S_p = S_p(K)$ denote the set of archimedean primes of $K$ resp. the set of primes of $K$ lying over $p$, and set $S^f := S \sm S_{\infty}$.

\begin{thm}\label{thm:main_boundary}
Assume the Greenberg conjecture holds for all involved number fields and for the fixed prime $p$. Let $K$ be a number field, $S \supseteq S_p \cup S_{\infty}$. Then 

\[ (G_{K,S},p) \rightsquigarrow (D_{\bap} \har G_{K,S})_{\bap \in S^f \sm S_p}. \]
\end{thm}

By all involved number fields we mean all finite subextensions of $K_S/K$. Let $\bN(S) := \bN \cap \caO_{K,S}^{\ast}$. From Theorem \ref{thm:main_boundary} and \cite{Iv} Theorem 1.1 and Proposition 4.1 we immediately deduce:

\begin{cor}\label{cor:intro:loccor_on_bdr_diff_conds}
Let $\Spec \caO_{K,S}$ be such that at least two rational primes lie in $\bN(S)$ and $S_{\infty} \subseteq S$. Let $p \in \bN(S)$ and let $\chi_p$ be the $p$-part of the cyclotomic character on $\Gal_{K,S}$. 
\begin{itemize}
\item[(i)] If the Leopoldt conjecture holds for all involved number fields and all primes and the Greenberg conjecture holds for all involved number fields for at least two primes in $\bN(S)$:
\[ G_{K,S} \rightsquigarrow (D_{\bap} \har G_{K,S})_{\bap \in S^f}. \]
\item[(ii)] If the Leopoldt conjecture holds for all involved number fields and all primes, then $G_{K,S} \rightsquigarrow \bN(S)$ and 
\[ (G_{K,S},\chi_p) \rightsquigarrow (D_{\bap} \har G_{K,S})_{\bap \in S^f}. \]
\item[(iii)] $(\Gal_{K,S},p,\chi_p) \rightsquigarrow (D_{\fp} \subseteq \Gal_{K,S})_{\fp \in S^f }$.
\end{itemize}
\end{cor}

In the rest of the article we consider the local correspondence \emph{on the curve}, i.e., for points $\not\in S$. More precisely, we will give a (unfortunately, not purely) group theoretic criterion for a subgroup $D \subseteq \Gal_{K,S}$ to be a decomposition subgroup of a prime in the spirit of Tamagawa's work \cite{Ta} in the case of curves over finite fields. Here is our main result. The important technical conditions (C2)-(C4) made there will be explained in Sections \ref{sec:Conditions_C1_C3_etc}, \ref{sec:anab_geom} below. Essentially, they ensure that $\Gal_{K,S}$ is sufficiently big and that the decomposition subgroups lie sufficiently independent inside it. Unfortunately, these conditions are still not known to hold for arithmetic fundamental groups. 

\begin{thm}\label{thm:main_thm}
Assume that the generalized Riemann hypothesis holds for all involved number fields. For $i=1,2$ let $\Spec \caO_{K_i,S_i}$ be an arithmetic curve. Let 
\[\sigma \colon \Gal_{K_1,S_1} \stackrel{\sim}{\rar} \Gal_{K_2,S_2}\] 

\noindent be a topological isomorphism of the fundamental groups. Assume that 

\begin{itemize}
\item[(1)] $S_i$ contains the archimedean primes and at least two rational primes are invertible on $\Spec \caO_{K_i,S_i}$, $K_{i,S_i}$ realizes locally at each $\fp \in S_i$ the maximal local extension\footnote{If $S \supseteq S_p \cup S_{\ell} \cup S_{\infty}$ and $S$ is defined over a totally real subfield, then this holds (cf. \cite{CC} Remark 5.3(i)).} and $\sigma$ induces a local correspondence on the boundary (cf. Corollary \ref{cor:intro:loccor_on_bdr_diff_conds})
\item[(2)] condition (C4) and either condition (C2) or condition (C3) hold for $K_i,S_i$ for $i = 1,2$\footnote{cf. Sections \ref{sec:Conditions_C1_C3_etc}, \ref{sec:anab_geom}.} 
\item[(3)] for any finite (sufficiently big) subextension $K_{1,S_1}/L_1/K_1$, one has $h_{L_1}\Regul_{L_1} = h_{L_2}\Regul_{L_2}$\footnote{In fact, it is enough to require a condition on the behavior of this quantity in infinite Galois subtowers $K_{i,S_i}/\cL_i/L_i/K_i$. Further, due to \cite{Iv} Proposition 4.2, it is equivalent to require equality of regulators.}.
% \item[(5)] for any \emph{infinite} normal subextension $K_{1,S_1}/\cL_1/L_1/K_1$, such that $\cL_1/L_1$ is ramified, one has:
% \[ \lim_{L_1} \frac{\log h_{L_1}(\Regul_{L_1} - \Regul_{L_2})}{g_{L_1/K_1}} = 0, \]
% \noindent where the limit is taken over finite subextensions $L_1$ of $\cL_1/K_1$ and $L_2/K_2$ is the field corresponding to $L_1$ via $\sigma$.
\end{itemize}

\noindent Then for any $K_{1,S_1}/L_1/K_1$ finite, $\sigma$ induces a bijection (the local correspondence map):
\[ \sigma_{\ast,L_1} \colon |\Spec \caO_{L_1,S_1}| \stackrel{\sim}{\rar} |\Spec \caO_{L_2,S_2}|, \] 

\noindent characterized by $\sigma(D_{\bap}) = D_{\sigma_{\ast}(\bap)}$, where $\sigma_{\ast}$ is the inverse limit over all $\sigma_{\ast,L_1}$. The maps $\sigma_{\ast,L_1}$ are compatible for varying $L_1$ and preserve the residue characteristic and the absolute inertia degrees of primes.
\end{thm}

This theorem gives an approach to the (weaker form of) Isom-conjecture of Grothendieck (cf. \cite{Gro} and \cite{NSW} 12.3.4,12.3.5) in the case of arithmetic curves. Indeed, as $\sigma_{\ast}$ preserves the absolute norm of primes, it also preserves the residue characteristics and the inertia degrees over $\bQ$. Thus by Chebotarev density theorem, one has the following corollary (which is proven exactly as Theorem 2 in \cite{Ne}).

\begin{cor}
For $i=1,2$ let $\Spec \caO_{K_i,S_i}$ be an arithmetic curve. If a local correspondence $\sigma_{\ast}$ as in Theorem \ref{thm:main_thm} is established, and $K_1/\bQ$ is normal, then $K_1 \cong K_2$.
\end{cor}

% We mention a different point of view to Theorem \ref{thm:main_thm}. Namely, assume for a number field $L$ the number of complex and real embeddings, its class number, (the absolute value of) its discriminant and $\sharp\mu(L)$ are known (in our set up, most of them can be deduced from the position inside $\Gal_{K,S}$ of the decomposition groups at primes in $S^f$). Then by the class number formula the knowledge of $\Regul_L$ is equivalent to the knowledge of the residue of the zeta-function $\zeta_L$ at $s = 1$. 

As $h_L\Regul_L$ is encoded in the residue at $s = 1$ of the zeta-function $\zeta_L$ of $L$, Theorem \ref{thm:main_thm} says essentially that $\Gal_{K,S}$ (under certain hypotheses, ensuring that it is sufficiently big) plus the family of the residues of the functions $\zeta_L$ for $K_S^{\bullet}/L/K$ at $s = 1$ give enough information to reconstruct $K$, at least when $K$ is normal over $\bQ$.

The main idea in the proof of Theorem \ref{thm:main_thm} is a group-theoretic characterization of decomposition subgroups of points on $\Spec \caO_{K,S}$ in the spirit of Tamagawa's work \cite{Ta}. Instead of using \'{e}tale cohomology and Lefschetz fixed point formula, which are not available in our case, we use the theorem of Tsfasman-Vl\u{a}du\c{t} \cite{TV} (and its further generalization by Zykin \cite{Zy}), which is a generalization of the Brauer-Siegel theorem. The application of Tsfasman-Vl\u{a}du\c{t} theorem is the place where we need the behavior of the regulators of the fields in question. While it is completely unclear how to recover these regulators from the fundamental group $\Gal_{K,S}$, the $p$-adic volume $\vol_p(K)$ of the unit lattice (for totally real fields, it is the $p$-adic norm of the $p$-adic regulator) can be recovered if $S_p \subseteq S$ (see Section \ref{sec:padic_regulator}). Unfortunately, the $p$-adic analogue of Brauer-Siegel and hence also of Tsfasman-Vl\u{a}du\c{t} fails (cf. \cite{Wa}), so our method would fail, if we replace $\Regul_K$ by $\vol_p(K)$. This do not exclude that there is still some similar way of using $\vol_p(K)$ to obtain information on decomposition subgroups.  

Note the following subtlety: in the application of Tsfasman-Vl\u{a}du\c{t}, we need to restrict attention to the maximal tame subextension, as infinitely wildly ramified towers are in general asymptotically bad and hence the invariants we work with get trivial. We were inspired to consider tamely ramified towers to obtain non-trivial invariants by the work of Hajir and Maire \cite{HM}. Nevertheless, part (2) of the assumptions of Theorem \ref{thm:main_thm} concerns the \emph{whole} group $\Gal_{K,S}$, and not only its maximal tame quotient. This is possible because we are able to deduce non-trivial invariants attached to an infinitely wildly ramified tower by considering limits of invariants of tame subtowers, cf. e.g. Proposition \ref{prop:elegant_comp_for_almost_tame_towers}. 

\subsection*{Notation}
Let us collect the notations used throughout the paper. Let $K$ be a number field, i.e., a finite extension of $\bQ$. Then $\Sigma_K$ is the set of all primes (archimedean or not) of $K$, $S_{\infty} := S_{\infty}(K)$ is the set of archimedean primes of $K$ and $S^f := S \sm S_{\infty}(K)$ for any set $S$ of primes of $K$. For $S,R \subseteq \Sigma_K$, $K_S^R$ is the maximal extension of $K$ unramified outside $S$ and completely split in $R$. For a prime $\fp$ of $K$ we write $\N\fp$ for its norm over $\bQ$. Further, $D_K, h_K, \Regul_K$ are the absolute discriminant, the class number, the regulator of $K$ and $g_K := \log |D_K|^{\frac{1}{2}}$ is the genus of $K$. For a finite ramified extension $L/K$ we set $g_{L/K} := \log |\N_{K/\bQ} D_{L/K}|^{\frac{1}{2}}$.

If $L/K$ is a Galois extension and $\bap$ is a prime of $L$, then $D_{\bap, L/K} \subseteq \Gal_{L/K}$ denotes the decomposition subgroup of $\bap$. If $\fp := \bap|_K$ is the restriction of $\bap$ to $K$, then we sometimes write $D_{\bap/\fp}$ or simply $D_{\bap}$ instead of $D_{\bap,L/K}$.

A prime always mean a non-archimedean prime.

\subsection*{Outline of the paper}

In Section \ref{sec:boundary_via_Greenberg_conjecture}, we consider the approach to the local correspondence at the boundary via Iwasawa theory. In Section \ref{sec:some_invs_of_infinite_towers} we show how to use results of Tsfasman-Vl\u{a}du\c{t} to deduce information on wildly ramified towers. In Section \ref{sec:char_of_dec_groups} we characterize decomposition subgroups inside $\Gal_{K,S}$ using certain invariants attached to intermediate fields. In Section \ref{sec:anab_geom} we prove our main result Theorem \ref{thm:main_thm} and in Section \ref{sec:padic_regulator} we show how to reconstruct the $p$-adic volume of the unit lattice from $\Gal_{K,S}$.

%************************************************************************************************************************************************************
%************************************************************************************************************************************************************

\begin{comment}
\subsection*{Outline of the paper}

In Section \ref{sec:Preliminaries} we recall the necessary preliminaries, do some technical computations and make a digression on how to reconstruct the $p$-adic volume of the unit lattice from $\Gal_{K,S}$. In Section \ref{sec:Small_norm_crits} we prove two criterions for subgroups of $\Gal_{K,S}^{\bullet}$ being related to primes of $K$. They are the main ingredients in the proof of Theorem \ref{thm:main_thm}, which follows in Section \ref{sec:anab_geom_proof_of_main_thm}. 
\end{comment}

%************************************************************************************************************************************************************
%************************************************************************************************************************************************************

\section{Decomposition subgroups of primes at the boundary} \label{sec:boundary_via_Greenberg_conjecture}

In this section, we use Iwasawa theory, to strengthen results from \cite{Iv} significantly under assumption of the Greenberg conjecture. We are going to prove Theorem \ref{thm:main_boundary} from the introduction.

\begin{repthm}{thm:main_boundary}
Assume the Greenberg conjecture holds for all involved number fields and for the fixed prime $p$. Let $K$ be a number field, $S \supseteq S_p \cup S_{\infty}$. Then 

\[ (G_{K,S},p) \rightsquigarrow (D_{\bap} \har G_{K,S})_{\bap \in S^f \sm S_p}. \]
\end{repthm}

\begin{proof}
By \cite{Iv} Proposition 2.4, the intersections of the $p$-Sylow subgroups inside the decomposition subgroups $D_{\bap/\fp}$ of primes $\fp \in S^f \sm S_p$ are not open in any of them. This implies that it is enough to prove the theorem after replacing $K$ by a finite subextension contained in $K_S$ (cf. \cite{Iv} Remark 1.2). Thus we may assume that $\mu_p \subseteq K$.

Let $K_{\infty}/K$ be the compositum of all $\bZ_p$-extensions of $K$. Then the Galois group $\Gamma := \Gal_{K_{\infty}/K}$ is isomorphic to $\bZ_p^d$ with $d = r_2 + 1 + \delta_p(K)$, where $\delta_p(K)$ is the Leopoldt defect of $K$. Let 

\[ \Lambda := \bZ_p\llbracket \Gamma \rrbracket \cong \bZ_p \llbracket T_1, \dots, T_d \rrbracket\] 

\noindent be the corresponding Iwasawa algebra. Introduce the following notations:

\begin{eqnarray*}
M_{\infty} &=& \text{maximal abelian pro-$p$-extension of $K_{\infty}$ unramified outside $S_p \cup S_{\infty}$} \\
M_{\infty,S} &=& \text{maximal abelian pro-$p$-extension of $K_{\infty}$ unramified outside $S$} \\
Y &=& \Gal_{M_{\infty}/K_{\infty}} \\
Y_S &=& \Gal_{M_{\infty,S}/K_{\infty}}.
\end{eqnarray*}

\noindent For a prime $\fp \in S$, let $\Gamma_{\fp} \subseteq \Gamma$ denote the decomposition subgroup of $\fp$. The next lemma shows how the Iwasawa-theoretic result \cite{NSW} 11.3.5 generalizes to $K_{\infty}/K$. Let $I(K_{\fp}(p)/K_{\fp})$ denote the inertia subgroup of the Galois group of the extension $K_{\fp}(p)/K_{\fp}$ and let $\bZ_p(1)$ denote the first Tate-twist and $\Ind_{\Gamma_{\fp}}^{\Gamma}$ the compact induction functor from compact $\Gamma_{\fp}$-modules to compact $\Gamma$-modules, as explained in \cite{NSW} in the text preceeding 11.3.5. %in the sense of \cite{NSW} footnotes on p.61 and p.63

\begin{lm}\label{lm:ex_seq_of_Lambda_modules}
There is a canonical exact sequence of $\Lambda$-modules

\begin{equation} \label{eq:ex_seq_of_Lambda_modules} 
0 \rar \bigoplus_{\fp \in (S^f \sm S_p)(K)} \Ind\nolimits_{\Gamma_{\fp}}^{\Gamma} \bZ_p(1) \rar Y_S \rar Y \rar 0. 
\end{equation}

\end{lm}

\begin{proof}
As the weak Leopoldt conjecture holds for the cyclotomic $\bZ_p$-extension, it also holds for $K_{\infty}/K$ (\cite{Ng} Corollary 2.9). Hence, exactly as in the proof of \cite{NSW} 11.3.5 we obtain the exact sequence

\[ 0 \rar \prolim_{U \subseteq \Gamma} \prod_{\fp \in (S^f \sm S_p)(L)} I(K_{\fp}(p)/K_{\fp})_{\Gal_{L,\fp}} \rar Y_S \rar Y \rar 0, \]

\noindent where the limit is taken over all open subgroups $U \subseteq \Gamma$, $L$ is the subfield of $K_{\infty}$ corresponding to $U$ and $\Gal_{L,\fp}$ is the decomposition subgroup of $\fp$ in the extension $K_{S_p \cup S_{\infty}}(p)/L$. As $K_{\infty}$ contains the cyclotomic $p$-extension of $K$, we have $\Gamma_{\fp} \cong \bZ_p$. We have $K_{\infty} \subseteq K_{S_p \cup S_{\infty}}(p)$, the decomposition subgroup of a prime $\fp \in S^f \sm S_p$ in $K_{S_p \cup S_{\infty}}(p)/K_{\infty}$ is trivial and $I(K_{\fp}(p)/K_{\fp}) \cong \bZ_p(1)$. Hence the lemma follows.
\end{proof}

\begin{comment}
For $\fp \in S_p(K)$, let $r_{\fp}$ denote the $\bZ_p$-rank of $\Gamma_{\fp}$. 

\begin{lm}\label{lm:rp_geq_2}
For $\fp \in S_p(K)$, we have $r_{\fp} \geq 2$.
\end{lm}
\begin{proof}
% We reduce to the case $K = \bQ(\mu_p)$ 
As $\mu_p \subseteq K$, there is a natural map $\Gal_{K,S_p \cup S_{\infty}}(p) \rar \Gal_{\bQ(\mu_p),S_p \cup S_{\infty}}(p)$ with open image $U$, which induces a surjection $\Gal_{K,S_p \cup S_{\infty}}(p)^{\ab} \tar U^{\ab}$, which in turn induces (after factoring out the $p$-torsion) a surjection $\Gamma \rar U^{ab}/(tors)$. Fix $\fp \in S_p(K)$. As any local $p$-extension of $\fp|_{\bQ(\mu_p)}$ which is realized by $\bQ(\mu_p)_{S_p \cup S_{\infty}}(p)/\bQ(\mu_p)$ is also realized by $K_{S_p \cup S_{\infty}}(p)/K$, the above map induces a map $D_{\fp, K_{S_p \cup S_{\infty}}(p)/K} \rar D_{\fp|_{\bQ(\mu_p)}, \bQ(\mu_p)_{S_p \cup S_{\infty}}(p)/\bQ(\mu_p)}$ with open image. (...) Now the statement of the lemma follows from the fact that in $\bQ(\mu_p)/\bQ$ there is precisely one prime lying over $\fp$ and the $\bZ_p$-rank of $\Gamma$ is $\geq 2$ (cf. \cite{McCal} proof of Theorem 26).
\end{proof}
\end{comment}
Observe that $(G_{K,S},p) \rightsquigarrow \Gamma, \Lambda$, the $\Lambda$-module $Y_S$. Define the $\Lambda$-module $V$ by 

\[ V := \ker(Y_S \rar Y) = \bigoplus\limits_{\fp \in (S^f \sm S_p)(K)}  \Ind\nolimits_{\Gamma_{\fp}}^{\Gamma} \bZ_p(1). \] 

\begin{lm}\label{lm:GKSp_determines_V}
$(G_{K,S},p) \rightsquigarrow$ the sequence \eqref{eq:ex_seq_of_Lambda_modules} of $\Lambda$-modules and, in particular, the $\Lambda$-module $V$.
\end{lm}
\begin{proof}
By \cite{LN} Theorem 3.2 and as $\mu_p \subseteq K$, the $\bZ_p$-rank of $\Gamma_{\fp}$ is $\geq 2$ for $\fp \in S_p(K)$, hence the conditions of \cite{LN} Proposition 3.6 (cf. also \cite{LN} Theorem 4.4) resp. \cite{McCal} Corollary 14 are satisfied and we deduce that the Greenberg conjecture is equivalent to $Y$ being $\Lambda$-torsion free. On the other side, $\bigoplus_{\fp \in (S^f \sm S_p)(K)} \Ind\nolimits_{\Gamma_{\fp}}^{\Gamma} \bZ_p(1)$ is a torsion $\Lambda$-module, hence it is exactly the $\Lambda$-torsion submodule of $Y_S$. This determines it intrinsically by $(G_{K,S},p)$.
\end{proof}

\begin{lm}\label{lm:GKSp_determines_sharp_Sf_sm_Sp}
$(G_{K,S},p) \rightsquigarrow \sharp (S^f \sm S_p)(K)$.
\end{lm}
\begin{proof}
Let $\chi$ be a $\bZ_p^{\ast}$-valued (continuous) character of $\Gamma$ and $\bZ_p(\chi)$ the corresponding $\Lambda$-module. By Frobenius reciprocity we have:

\[ \Hom_{\Gamma}(V, \bZ_p(\chi)) = \bigoplus_{\fp \in (S^f \sm S_p)(K)} \Hom_{\Gamma}(\Ind\nolimits_{\Gamma_{\fp}}^{\Gamma} \bZ_p(1), \bZ_p(\chi)) 
= \bigoplus_{\fp \in (S^f \sm S_p)(K)} \Hom_{\Gamma_{\fp}} (\bZ_p(1), \bZ_p(\chi)). \]

\noindent Thus $\rank_{\bZ_p} \Hom_{\Gamma}(V, \bZ_p(\chi)) \leq {\sharp}(S^f \sm S_p)(K)$ with equality if and only if $\chi|_{\Gamma_{\fp}}$ agrees with the cyclotomic character on $\Gamma_{\fp}$. Thus 

\[ \sharp (S^f \sm S_p)(K) = \max_{\chi} \rank_{\bZ_p} \Hom_{\Gamma}(V, \bZ_p(\chi)), \]

\noindent where $\chi$ goes through all continuous $\bZ_p^{\ast}$-valued characters of $\Gamma$. This and Lemma \ref{lm:GKSp_determines_V} finish the proof.
\end{proof}

Now we can prove Theorem \ref{thm:main_boundary} using methods from \cite{Iv}. Applying Lemma \ref{lm:GKSp_determines_sharp_Sf_sm_Sp} to open subgroups $U \subseteq \Gal_{K,S}$, we obtain the numbers $\sharp (S^f \sm S_p)(L)$ for all intermediate subfields $K_S/L/K$. Let $U \subseteq \Gal_{K,S}$ be an open subgroup. Let 

\begin{equation}\label{eq:def_of_V0} 
V_0 := \bigcap_{V \subseteq U} V,
\end{equation}

\noindent where the intersection is taken over all open normal subgroups $V$ of $U$, such that $U/V$ is abelian of exponent $p$ and such that $\sharp (S^f \sm S_p)(K_S^V) = (V:U)\sharp (S^f \sm S_p)(K_S^U)$. Thus $K_S^{V_0}/K_S^U$ is the maximal abelian extension of $K_S^U$ of exponent $p$, which is unramified outside $S$ and completely split in $S^f \sm S_p$. Let $L = K_S^U$ and consider the natural surjection 
\[ t_{p,U} \colon U \tar U/V_0 = \Gal(L_S^{S^f \sm S_p}/L)^{\ab}/p. \]

As in \cite{Iv} Definition 2.1, let us say that a pro-finite group is of \emph{$p$-decomposition group} if it is a non-abelian pro-$p$ Demushkin group of rank $2$. Explicitly, such a group is necessarily isomorphic to $\bZ_p \ltimes \bZ_p$, such that the corresponding map $\bZ_p \rar \Aut(\bZ_p) \cong \bZ_p^{\ast}$ is injective. We have a variation of \cite{Iv} Proposition 3.5.

\begin{lm}\label{lm:criterion_for_being_pSyl_of_dec_group}
Let $H$ be a group of $p$-decomposition type inside $\Gal_{K,S}$. The following are equivalent:
\begin{itemize}
\item[(i)]  $H \subseteq D_{\bap}$ for some $\bap \in S^f \sm S_p$
\item[(ii)] For any $U \subseteq \Gal_{K,S}$ open: $H \subseteq U \Rar H \subseteq \ker(t_{p,U} \colon U \tar U/V_0)$, where $V_0$ is as in \eqref{eq:def_of_V0}.
\end{itemize}

\noindent The prime $\bap$ in (i) is unique.
\end{lm}

\begin{proof}
(i) $\Rar$ (ii): Assume $H \subseteq D_{\bap}$ for some $\bap \in S^f \sm S_p$. By construction, $K_S^{V_0}/K_S^U$ is completely decomposed in $\bap|_{K_S^U}$. Thus $H \subseteq V_0$.

\noindent (ii) $\Rar$ (i): Let $H$ be given and let $H \subseteq U \subseteq \Gal_{K,S}$ be an open subgroup. Let $V_0$ be as in \eqref{eq:def_of_V0}. Then $K_S^{V_0}/K_S^U$ contains the maximal unramified abelian extension of $K_S^U$ of exponent $p$, which is completely split in $S$. Thus $S$-version of the principal ideal theorem (as in the proof of the (a) $\Rar$ (c)-part of \cite{Iv} Proposition 3.5) implies:

\[ \dirlim_{H \subseteq U \subseteq \Gal_{K,S}} \Cl_S(K_S^U)/p = 0. \]

\noindent Now the proof can be finished as in \cite{Iv} Proposition 3.5.
\end{proof}

For $\bap \in (S^f \sm S_p)(K_S)$, the $p$-Sylow subgroups of $D_{\bap}$ are exactly the subgroups of $p$-decomposition type. Thus by Lemmas \ref{lm:GKSp_determines_sharp_Sf_sm_Sp} and \ref{lm:criterion_for_being_pSyl_of_dec_group} we obtain for any open subgroup $U \subseteq \Gal_{K,S}$:

\[ (\Gal_{K,S}, p) \rightsquigarrow \Syl_p(U,S^f \sm S_p) := \left\{ H \subseteq U \colon \begin{array}{cl} \text{$H$ is a $p$-Sylow subgroup of $D_{\bap,K_S/K_S^U}$ with $\bap \in S^f \sm S_p$}   \end{array} \right\} \]

\noindent Now the proof of Theorem \ref{thm:main_boundary} can be finished as in \cite{Iv} Section 3.4. 
\end{proof}

%************************************************************************************************************************************************************
%************************************************************************************************************************************************************

\section{Invariants of infinite Galois towers of number fields} \label{sec:some_invs_of_infinite_towers}

\subsection{Preliminaries: Tsfasman-Vl\u{a}du\c{t} theorem} \label{sec:prelims_TV_thm} \mbox{}

We will make use of the generalized Brauer-Siegel theorem proven by Tsfasman-Vl\u{a}du\c{t} \cite{TV} and extended further by Zykin \cite{Zy}, so let us recall their results. Tsfasman and Vl\u{a}du\c{t} work with asymptotically exact \emph{families} of number fields. We need only the special case of infinite number fields (equivalently, of towers of number fields). Let $\cK$ be an infinite number field, i.e., an algebraic extension of $\bQ$ of infinite degree. We can choose a tower 

\[ \bQ \subsetneq K_0 \subsetneq K_1 \subsetneq K_2 \subsetneq \dots \subsetneq \cK \] 

\noindent of subfields, finite over $\bQ$, such that $\cK = \bigcup_{n=1}^{\infty} K_n$. Such towers are clearly not unique. All results in this section (and also later on) depend only on $\cK$, not on the choice of the tower. For a number field $K$, set
\begin{itemize}
\item $\Phi_{\alpha}(K) := \{ \fp \in \Sigma_K \colon \N\fp = \alpha \}$ if $\alpha$ is a rational prime power
\item $\Phi_{\alpha}(K)$ is the set of real resp. complex primes of $K$, if $\alpha = \bR$ resp. $\bC$.
\end{itemize}

For $\cK,K_n$ as above and $\alpha$ either a rational prime power or $\bR$ or $\bC$, define 

\[ \phi_{\alpha}(\cK) := \lim_{n \rar \infty} \frac{\sharp\Phi_{\alpha}(K_n)}{g_{K_n}}. \]

\noindent Then \cite{TV} Lemma 2.4 shows that this limit exists for each $\alpha$ (i.e., $\{K_n\}_n$ is an \emph{asymptotically exact} family). Moreover, \cite{TV} Lemma 2.5 shows that these limits depend only on $\cK$ and $\alpha$ and not on the choice of the tower. Further, $\cK$ is called \emph{asymptotically good}, if there exists an $\alpha$ such that $\phi_{\alpha}(\cK) \neq 0$, and \emph{asymptotically bad} otherwise. Also define

\[ \BS(\cK) := \lim_{n \rar \infty} \frac{\log (h_{K_n}\Regul_{K_n}) }{g_{K_n}}. \]

\noindent (We will see in a moment that this notation makes sense). A finite extension $L/K$ is called \emph{almost normal}, if there are subextensions $L = K_n \supseteq K_{n-1} \supseteq \dots \supseteq K_0 = K$, such that $K_{i+1}/K_i$ is normal. An infinite extension $\cK/K$ is called almost normal, if it is the limit of a tower of finite almost normal extensions. The following generalization of the classical Brauer-Siegel theorem is shown by Tsfasman-Vl\u{a}du\c{t} in the asymptotically good case and also under GRH and then by Zykin in the remaining asymptotically bad unconditional case.

\begin{thm}[\cite{TV} GRH Corollary D, Corollary F; \cite{Zy} Corollary 3] \label{thm:TVZ} Let $\cK$ be an infinite number field and let $\{K_n\}_n$ be a tower of subextensions of $\cK$, such that $K_n \subsetneq K_{n+1}$ for all $n$ and $\cK = \bigcup_{i=0}^{\infty} K_n$. Assume that at least one of the following holds:
\begin{itemize}
\item each $K_n$ is almost normal over $\bQ$, or
\item generalized Riemann hypothesis holds for each $K_n$.
\end{itemize}

\noindent Then $\BS(\cK)$ exists, depends only on $\cK$, not on $K_n$'s and 
\[ \BS(\cK) = 1 + \sum_q \phi_q(\cK) \log\frac{q}{q-1} - \phi_{\bR}(\cK)\log 2 - \phi_{\bC}(\cK)\log 2\pi, \]
\noindent where the sum is taken over all rational prime powers. 
\end{thm}

Also observe that by \cite{TV} Theorem H (and GRH Theorem G) $\BS(\cK)$ is always finite. The term on the right side in Theorem \ref{thm:TVZ} is closely related with the value at $1$ of the zeta function of $\cK$ (which is defined in \cite{TV}). For an infinite number field $\cK$, satisfying one of the conditions from the theorem, we define:

\begin{equation} \label{eq:lambda_of_tower}
\lambda(\cK) := \BS(\cK) - 1 + \phi_{\bR}(\cK)\log 2 + \phi_{\bC}(\cK)\log 2\pi = \sum_q \phi_{q}(\cK) \log \frac{q}{q-1}.
 \end{equation}

% \noindent Moreover, assume we are given an infinite Galois extension $\cK/K$ of number fields, with $K$ finite over $\bQ$. Let $G$ be the corresponding profinite Galois group. If $G \supseteq U \triangleright H$ are two subgroups such that $U$ is open in $G$ and $H$ is normal, closed and not open in $U$, let $\cL := \cK^H, L := \cK^U$ be the corresponding fixed fields. Then we will write 
% \begin{equation*}
%  \lambda(U;H) := \lambda(\cL) \qquad \text{and} \qquad \phi_{\alpha}(U;H) := \phi_{\alpha}(\cL).
% \end{equation*}

% \noindent Then $\lambda(U;H),\phi_{\alpha}(U;H)$ depend only on $H$, not on $U$.

For an infinite extension $\cK/K$ with $\cK = \bigcup_{n > 0} K_n$ and $K_{n+1} \supseteq K_n \supseteq K$ for all $n > 0$, define:

\[ \mu(\cK/K) := \lim\limits_{n \rar \infty} \frac{[K_n:K]}{g_{K_n}}. \] 

% \smallskip
% \noindent Obviously, one has 
% \[ \mu(\cK/K) = [L:K] \mu(\cK/L)\] 
% \noindent for $\cK/L/K$ if $L/K$ is finite. 

\noindent The following two lemmas below are closely related to \cite{HM} Lemma 5, Definition 6 and the invariant $\mu$ is related to the root discriminant ${\rm rd}_K := |D_{K/\bQ}|^{\frac{1}{[K:\bQ]}}$. 

\begin{lm} \label{lm:mu_is_strictly_pos_in_tame_towers}
Let $K$ be a number field and $\cK/K$ an infinite Galois extension. Let $(K_n)_{n=0}^{\infty}$ be a tower corresponding to $\cK/K$, consisting of Galois subextensions with $K = K_0$. Then 
\[ \mu(\cK/K) = \lim\limits_{n \rar \infty} \frac{[K_n:K]}{g_{K_n}} \geq 0 \] 

\noindent exists, is finite and depends only on $\cK/K$, not on the choice of the tower. If moreover, $\cK/K$ is unramified outside a finite set $S$ of primes of $K$ and only tamely ramified in $S$, then $\mu(\cK/K) > 0$.
\end{lm}

\begin{proof} 
This follows from the last statement of Lemma \ref{lm:disccomp_1} below.
\end{proof}

\begin{lm} \label{lm:disccomp_1}
Let $L/K$ be a finite Galois extension, which is unramified outside $S_{\infty}$ a set $S^f = \{\fp_1, \dots, \fp_r\}$ of primes of $K$, and let $\fp_i\caO_{L} = \prod_{j=1}^{g_i} \fP_{ij}^{e_i}$, such that $[L:K] = e_i f_i g_i$, where $f_i$ is the inertia degree of $\fP_{ij}/\fp_i$ for some (any) $j$. Then 

\[ N_{K/\bQ}D_{L/K} = \prod_{i = 1}^r \N\fp_i^{[L:K](1 - \frac{1}{e_i} + \frac{\beta_i}{e_i} )}, \] 

\noindent with some $\beta_i \geq 0$ and $\beta_i = 0 \LRar L/K$ is tamely ramified in $\fp_i$. Moreover, one has

\[ \frac{g_L}{[L:K]} = g_K + \frac{1}{2} \sum_{i = 1}^r (1 - \frac{1}{e_i} + \frac{\beta_i}{e_i}) \log \N\fp_i.  \] 
\end{lm}

\begin{proof} Let $\delta_{L/K}$ be the different of $L/K$. Then the $\fP_{ij}$-valuation of $\delta_{L/K}$ is $e_i - 1 + \sum_{u=1}^{\infty} (\sharp D_{\fP_{ij}/\fp_i, u} - 1)$, where $D_{\fP_{ij}/\fp_i, u}$ denote the higher ramification subgroups in the lower numbering (\cite{Se} Chap. IV Proposition 4). Let $\beta_i := \sum_{u=1}^{\infty} (\sharp D_{\fP_{ij}/\fp_i, u} - 1)$. As $D_{L/K} = \N_{L/K}\delta_{L/K}$, we get

\[ \N_{K/\bQ}D_{L/K} = \N_{K/\bQ}\N_{L/K}\delta_{L/K} = \N_{L/\bQ}\delta_{L/K} = \prod_{i=1}^r \N\fp_i^{g_i f_i (e_i - 1 + \beta_i)} = \prod_{i=1}^r \N\fp_i^{[L:K](1 - \frac{1}{e_i} + \frac{\beta_i}{e_i} )}. \]

\noindent It is clear that $\beta_i = 0 \LRar L/K$ is tamely ramified in $\fp_i$. Now taking the logarithm of the formula $D_{L/\bQ} = D_{K/\bQ}^{[L:K]}\N_{K/\bQ}D_{L/K}$ gives the last statement of the lemma.
\end{proof}

%************************************************************************************************************************************************************
%************************************************************************************************************************************************************

\subsection{Relative version of the invariants}\label{sec:rel_versions_of_invariants} \mbox{}

Let $L/K$ be a finite extension. Having no chance to control neither $D_{K/\bQ}$, nor the wild part of $D_{L/K}$, we are still able to control the tame part of the norm of the relative discriminant $|N_{K/\bQ} D_{L/K}|$ (by 'control' we mean, that the information is encoded in the Galois groups). Therefore, we introduce relative versions of invariants from Section \ref{sec:prelims_TV_thm}. For $L/K$ finite and ramified, set 

\[g_{L/K} := \log |\N_{K/\bQ} D_{L/K}|^{\frac{1}{2}}. \]

\noindent Then $g_L = g_{L/\bQ} = [L:K]g_K + g_{L/K}$. Assume now the tower $\cK = \bigcup_n K_n/K$ is ramified, i.e., $g_{K_n/K} > 0$ for $n \gg 0$). We define:

\begin{eqnarray*} 
\mu_{\rel}(\cK/K) &:=& \lim_n \frac{[K_n:K]}{g_{K_n/K}} \\
\phi_{\alpha,\rel}(\cK/K) &:=& \lim_n \frac{\sharp\Phi_{\alpha}(K_n)}{g_{K_n/K}} \\
\BS_{\rel}(\cK/K) &:=& \lim_n \frac{\log(h_{K_n}\Regul_{K_n})}{g_{K_n/K}}  \\
\lambda_{\rel}(\cK/K) &:=& \BS_{\rel}(\cK/K) - 1 + \phi_{\bR,\rel}(\cK/K) \log 2 + \phi_{\bC,\rel}(\cK/K) \log 2\pi.
\end{eqnarray*}

\begin{lm}\label{lm:rel_version_vs_absol_version} Let $\cK/K$ be a ramified tower. Then
\begin{itemize} 
\item[(i)]   $\mu_{\rel}(\cK/K), \phi_{\alpha,\rel}(\cK/K), \BS_{\rel}(\cK/K)$ and $\lambda_{\rel}(\cK/K)$ exist and are independent of the choice of the $K_n$'s.
\item[(ii)] One has $\mu_{\rel}(\cK/K) < \infty$ and $\mu(\cK/K) = 0 \LRar \mu_{\rel}(\cK/K) = 0$. 
\item[(iii)] For $\ast \in \{ \mu,\phi_{\alpha},\BS \}$ one has: 

\[ \ast_{\rel}(\cK/K) =  \ast(\cK/K) (1 + g_K \mu_{\rm rel}(\cK/K)).\] 

% \item[(ii)] Assume $\mu(U;H), \mu(U;1)$ are non-zero. Then the following holds:

% \[ \frac{\lambda(U;H)}{\mu(U;H)} - \frac{\lambda(U;1)}{\mu(U;1)} = \frac{\lambda_{\rel}(U;H)}{\mu_{\rel}(U;H)} - \frac{\lambda_{\rel}(U;1)}{\mu_{\rel}(U;1)}. \]
\end{itemize}
\end{lm}

\begin{proof} For $n \gg 0$, $g_{K_n/K}$ is defined and we have: 

\[ \frac{g_{K_n}}{[K_n:K]} = \frac{g_{K_n/K} + [K_n:K]g_K}{[K_n:K]} = \frac{g_{K_n/K}}{[K_n:K]} + g_K. \]

\noindent As for $n \rar \infty$ the limit on the left side exists and is $> 0$, also the limit on the right side exists and is $> 0$. Take $n_0$ big enough such that $K_{n_0}/K$ is ramified, then a similar computation shows that $\frac{g_{K_n/K}}{[K_n:K]} \geq g_{K_{n_0}/K}$, in particular the limit of $\frac{g_{K_n/K}}{[K_n:K]}$ is $> 0$. Going to this limit we obtain:

\begin{equation}\label{eq:mu_and_mu_rel}
\mu(\cK/K)^{-1} = \mu_{\rel}(\cK/K)^{-1} + g_K.
\end{equation}

\noindent The left hand side is $\infty$ if and only if the right hand side is. Thus $\mu(\cK/K) = 0 \LRar \mu_{\rel}(\cK/K) = 0$. This completes the proof of (ii). Also (iii) for $\ast = \mu$ follows from \eqref{eq:mu_and_mu_rel}. Further, as $\mu(\cK/K)$ depends only on $\cK/K$, not on $K_n$'s, the same holds for $\mu_{\rel}(\cK/K)$. We compute:

\[ (1 + g_K \mu_{\rel}(\cK/K)) \phi_{\alpha}(\cK/K) = (1 + g_K \lim_n \frac{[K_n:K]}{g_{K_n/K}} ) \lim_n \frac{\sharp \Phi_{\alpha}(K_n)}{g_{K_n}} = \lim_n \frac{\sharp \Phi_{\alpha}(K_n)}{g_{K_n/K}} = \phi_{\alpha, \rel}(\cK/K). \]

\noindent As the limits in the second term exist and are independent of the choice of the tower, the same holds true for the limit defining $\phi_{\alpha,\rel}$. Analogously, we deduce existence and independence of the $K_n$'s of $\BS_{\rel}$ and part (iii) for $\ast = \BS$. Now, existence of $\lambda_{\rel}$ and its independence of the choice of $K_n$'s follows then from the definition of $\lambda_{\rel}$. This finishes the proof of the lemma. \qedhere
% \[ (1 + g_K \mu_{\rel}(\cK/K))\BS(\cK/K) = \BS_{\rel}(\cK/K). \]
\end{proof}

For a Galois tower $\cK/K$ and for $\alpha$ either a rational prime power or $\bR,\bC$ define

\[ \psi_{\alpha}(\cK/K) := \lim\limits_{n \rar \infty} \frac{\sharp\Phi_{\alpha}(K_n)}{[K_n:K]}. \]

% \noindent This limit exists and is independent of the choice of $K_n$'s 
\noindent Assume an infinite tower $\cK/K$ is given with $0 < \mu_{\rel}(\cK/K) < \infty$. Then we have:

\[ \psi_{\alpha}(\cK/K) = \lim\limits_{n \rar \infty} \frac{\sharp\Phi_{\alpha}(K_n)}{[K_n:K]} = \left( \lim_{n \rar \infty} \frac{\sharp\Phi_{\alpha}(K_n)}{g_{K_n/K}} \right) \cdot \left( \lim\limits_{n \rar \infty} \frac{[K_n:K]}{g_{K_n/K}} \right)^{-1} = \frac{\phi_{\alpha,\rel}(\cK/K)}{\mu_{\rel}(\cK/K)}  \]

\noindent (in particular, $\psi_{\alpha}(\cK/K)$ exists and is independent of the choice of $K_n$'s in this case). 

\begin{prop} \label{prop:lamdba_mu_rel_ratio} 
Let $\cK/K$ be a tower with $0 < \mu_{\rel}(\cK/K) < \infty$. Then:

\begin{equation} \label{eq:lamdba_mu_rel_ratio} 
\frac{\lambda_{\rel}(\cK/K)}{\mu_{\rel}(\cK/K)} = g_K + \sum_q \psi_q(\cK/K) \log\frac{q}{q-1}. 
\end{equation}

\noindent Moreover, both sides are non-negative and $< \infty$.
\end{prop}

\begin{proof}
During the proof, for $\ast \in \{\mu,\BS,\phi_{\alpha},\lambda \}$, $\ast$ always means $\ast(\cK)$ and $\ast_{\rel}$ always means $\ast_{\rel}(\cK/K)$. We compute

\begin{eqnarray*}
\lambda_{\rel} + 1 &=&  \BS_{\rel} + \phi_{\bR,\rel} \log 2 + \phi_{\bC,\rel} \log 2\pi \\
&=& (1 + g_K \mu_{\rel})(\BS + \phi_{\bR}\log 2 + \phi_{\bC}\log 2\pi) \\
&=& (1 + g_K \mu_{\rel})(\lambda + 1) \\
&=& (1 + g_K \mu_{\rel}) (1 + \sum_q \phi_q \log \frac{q}{q-1}) \\
&=& 1 + g_K \mu_{\rel} + \sum_q \phi_{q,\rel} \log \frac{q}{q-1}.
\end{eqnarray*}

\noindent Subtracting $1$, dividing by $\mu_{\rel}$ and applying the definition of $\psi_q$, we deduce \eqref{eq:lamdba_mu_rel_ratio}. Now the right hand side of \eqref{eq:lamdba_mu_rel_ratio} is clearly $\geq 0$. To show that the left hand side is $< \infty$, it is enough to show that $\lambda_{\rel} < \infty$. Therefore, consider the equality 

\[ \lambda_{\rel} + 1 = (1 + g_K \mu_{\rel})(\lambda + 1), \] 

\noindent from the proof of Proposition \ref{prop:lamdba_mu_rel_ratio}. By assumption we have $1 + g_K \mu_{\rel} < \infty$. Further, $\lambda < \infty$ follows from the upper bounds on $\BS$ proven in \cite{TV} Theorem H (cf. also GRH Theorem G).
\end{proof}

%************************************************************************************************************************************************************
%************************************************************************************************************************************************************

\subsection{Galois towers with non-vanishing $\mu_{\rel}$} \mbox{}

Now we evaluate the right hand side of \eqref{eq:lamdba_mu_rel_ratio} in terms of Galois theory and norms of primes. Note that for a Galois extension $\cK/K$ and a prime $\fp$ of $K$, the degree (and the ramification index and the inertia degree) of the local extension $\cK_{\bap}/K_{\fp}$ does not depend on the choice of the prolongation $\bap$ of $\fp$ to $\cK$. This justifies the following definition.

\begin{Def}
For a Galois extension $\cK/K$ and a prime $\fp$ of $K$ let $f_{\fp}(\cK/K)$ resp. $e_{\fp}(\cK/K)$ denote the inertia degree resp. the ramification index of a (any) prolongation of $\fp$ to $\cK$. For an infinite Galois extension $\cK/K$, we set

\[ \beta(\cK/K) := \sum_{\fp \in \Sigma_K^f} \frac{1}{e_{\fp}(\cK/K)f_{\fp}(\cK/K)} \log \frac{\N\fp^{f_{\fp}(\cK/K)}}{\N\fp^{f_{\fp}(\cK/K)} - 1 } \]

\noindent (here and later, the sums are taken only over primes with finite absolute degree $e_{\fp}f_{\fp}$ in the tower $\cK/K$).
\end{Def}

Note that all summands are positive real numbers and if the sum converges, the convergence is absolute. In the following proposition and its proof we omit $\cK/K$ from notation and write $\beta$ instead of $\beta(\cK/K)$, etc.

\begin{prop} \label{prop:normal_case_invariants}
Let $\cK/K$ be an infinite Galois extension such that $0 < \mu_{\rel} < \infty$. Then

\[ \frac{\lambda_{\rel}}{\mu_{\rel}} = g_K + \beta. \]
\end{prop}

\begin{proof}
For a rational prime power $q$, a finite extension $L/K$ and $\fp \in \Sigma_K^f$, let 

\[ \Phi_{q,\fp}(L) = \{ \fq \in \Sigma_L^f \colon \N\fq = q, \fq \text{ lies over $\fp$} \}. \]

\noindent Clearly, for given $q$ and $L$, the set of all $\fp \in \Sigma_K^f$ such that $\Phi_{q,\fp}(L) \neq \emptyset$ is finite and contained in $S_p(K <)$. Using this, we compute:

\[ \psi_q = \lim_n \frac{\sharp \Phi_q(K_n)}{[K_n:K]} = \lim_n \sum_{\fp \in \Sigma_K^f} \frac{\sharp \Phi_{q,\fp}(K_n)}{[K_n:K]} = \sum_{\fp \in \Sigma_K^f} \lim_n \frac{\sharp \Phi_{q,\fp}(K_n)}{[K_n:K]} = \sum_{\substack{\fp \in \Sigma_K^f \\ e_{\fp}f_{\fp} < \infty \\ \N \fp^{f_{\fp}} = q }} \frac{1}{e_{\fp}f_{\fp}}. \]

\noindent Thus by Proposition \ref{prop:lamdba_mu_rel_ratio} we get:

\begin{eqnarray*} 
\frac{\lambda_{\rel}}{\mu_{\rel}} - g_K &=& \sum_q \psi_q \log \frac{q}{q-1} = \sum_q (\sum_{\substack{\fp \in \Sigma_K^f \\ e_{\fp}f_{\fp} < \infty \\ \N \fp^{f_{\fp}} = q }} \frac{1}{e_{\fp}f_{\fp}}) \log \frac{q}{q-1} \\ 
&=& \sum_{\substack{\fp \in \Sigma_K^f \\ e_{\fp}f_{\fp} < \infty } } \frac{1}{e_{\fp}f_{\fp}} \log \frac{\N\fp^{f_{\fp}}}{\N\fp^{f_{\fp}} - 1} = \beta. \qedhere
\end{eqnarray*}
\end{proof}

%************************************************************************************************************************************************************
%************************************************************************************************************************************************************

\subsection{Galois towers with arbitrary $\mu_{\rel} < \infty$} \mbox{}

Let us deduce invariants of wildly ramified towers. %(with $\mu_{\rel}(\cK/K) < \infty$) 

\begin{prop}\label{prop:elegant_comp_for_almost_tame_towers}
Let $\cK/K$ be an infinite Galois extension with $\mu_{\rel}(\cK/K) < \infty$. Let $H \subseteq \Gal_{\cK/K}$ be a closed normal subgroup with fixed field $\cK^H$. Assume for any finite Galois subextension $\cK^H/L/K$, we are given an ascending family $(\cL_n)_{n = 0}^{\infty}$ of infinite Galois subextensions $\cK/\cL_n/L$. Write $\cL_{\infty} := \bigcup_{n = 0}^{\infty} \cL_n$, $\cL_n^H := \cL_n \cap \cK^H$ for $n < \infty$ and $\cL_{\infty}^H := \bigcup_n \cL_n^H$. Assume that for all $L$ we have:
\begin{itemize}
\item $\cL_n^H/L$ is infinite, $0 < \mu_{\rel}(\cL_n/L) < \infty$  and $0 < \mu_{\rel}(\cL_n^H/L) < \infty$ for $0 \ll n < \infty$.
\end{itemize}

\noindent Then for any finite Galois subextension $\cK^H/L/K$, we have (in particular, the limit on the left side exists):

\[ \lim_n \left( \frac{\lambda_{\rel}(\cL_n^H/L)}{\mu_{\rel}(\cL_n^H/L)} - \frac{\lambda_{\rel}(\cL_n/L)}{\mu_{\rel}(\cL_n/L)} \right) = \beta(\cL_{\infty}^H/L) - \beta(\cL_{\infty}/L). \]

% \sum_{\fp \in \Sigma_L^f} \frac{1}{e_{\fp}(\cL_{\infty}^H/L) f_{\fp}(\cL_{\infty}^H/L)} \log \frac{\N\fp^{f_{\fp}(\cL_{\infty}^H/L)}}{\N\fp^{f_{\fp}(\cL_{\infty}^H/L)} - 1}

\noindent Moreover, assume $\cL_{\infty}^H/K$ is Galois for any $L$ and $\bigcup_{\cK^H/L/K} \cL_{\infty}$ has no primes of finite norm. Then we have

\[ \lim_{\cK^H/L/K} \frac{1}{[L:K]} (\beta(\cL_{\infty}^H/L) - \beta(\cL_{\infty}/L)) =  \beta(\cK^H/K). \]
% \sum_{\fp \in \Sigma_K^f} \frac{1}{e_{\fp}(\cK^H/L) f_{\fp}(\cK^H/L)} \log \frac{\N\fp^{f_{\fp}(\cK^H/L)}}{\N\fp^{f_{\fp}(\cK^H/L)} - 1}
\end{prop}

\begin{proof}
To prove the first statement, fix some finite Galois subextension $\cK^H/L/K$. As $0 < \mu_{\rel}(\cL_n/L) < \infty$, we know that $\beta(\cL_n/L)$ exists by Proposition \ref{prop:normal_case_invariants}. Moreover, for $m \leq n$, we have $\beta(\cL_m/L) \geq \beta(\cL_n/L) \geq 0$, hence $\beta(\cL_n/L)$ form a monotone decreasing, bounded from below by $0$, hence convergent sequence and we have 

\begin{eqnarray} \nonumber
\lim_n \beta(\cL_n/L) &=& \lim_n \sum_{\fp \in \Sigma_L^f} \frac{1}{e_{\fp}(\cL_n/L)f_{\fp}(\cL/L)} \log \frac{\N\fp^{f_{\fp}(\cL_n/L)}}{\N\fp^{f_{\fp}(\cL_n/L)} - 1 } \\ \label{eq:beta_infty_lim_beta_n}
&=& \sum_{\fp \in \Sigma_L^f} \lim_n \frac{1}{e_{\fp}(\cL_n/L)f_{\fp}(\cL_n/L)} \log \frac{\N\fp^{f_{\fp}(\cL_n/L)}}{\N\fp^{f_{\fp}(\cL_n/L)} - 1 }  \\
&=& \sum_{\fp \in \Sigma_L^f} \frac{1}{e_{\fp}(\cL_{\infty}/L)f_{\fp}(\cL_{\infty}/L)} \log \frac{\N\fp^{f_{\fp}(\cL_{\infty}/L)}}{\N\fp^{f_{\fp}(\cL_{\infty}/L)} - 1 } = \beta(\cL_{\infty}/L) \nonumber.
\end{eqnarray}

\noindent In the second equation we are allowed to interchange the limit with the sum by the dominated convergence theorem, as the summands form a monotone decreasing sequence for $n \rar \infty$. In particular, $\beta(\cL_{\infty}/L)$ exists. Analogously, we deduce that $\beta(\cL_{\infty}^H/L)$ exists and that 
\begin{equation} \label{eq:beta_infty_lim_beta_n_cap_H}
\beta(\cL_{\infty}^H/L) = \lim_n \beta(\cL_n^H/L). 
\end{equation}

\noindent For any $n \geq 0$, we have by Proposition \ref{prop:normal_case_invariants}:

\begin{equation}\label{eq:diff_of_lambda_mu_quots_wild_normal}
\frac{\lambda_{\rel}(\cL_n^H/L)}{\mu_{\rel}(\cL_n^H/L)} - \frac{\lambda_{\rel}(\cL_n/L)}{\mu_{\rel}(\cL_n/L)} = \beta(\cL_n^H/L) - \beta(\cL_n/L).
\end{equation}

\noindent Since by \eqref{eq:beta_infty_lim_beta_n} and \eqref{eq:beta_infty_lim_beta_n_cap_H} the limit for $n \rar \infty$ of the right side of \eqref{eq:diff_of_lambda_mu_quots_wild_normal} exists, also the limit of the left side exists and we have

\[ \lim_n \left( \frac{\lambda_{\rel}(\cL_n^H/L)}{\mu_{\rel}(\cL_n^H/L)} - \frac{\lambda_{\rel}(\cL_n/L)}{\mu_{\rel}(\cL_n/L)} \right) = \lim_n \beta(\cL_n^H/L) - \lim_n \beta(\cL_n/L) = \beta(\cL_{\infty}^H/L) - \beta(\cL_{\infty}/L). \]

\noindent The first statement of the proposition follows. The second statement follows from Lemma \ref{lm:cont_of_beta} which has to be applied twice: first to $\overline{\cL} := \bigcup_{\cK^H/L/K} \cL_{\infty}$ and $\cL_L := \cL_{\infty}$ and then to $\overline{\cL} := \bigcup_{\cK^H/L/K} \cL_{\infty}^H = \cK^H$ and $\cL_L := \cL_{\infty}^H$. \end{proof}

We have the following continuity property of $\beta$:

\begin{lm}\label{lm:cont_of_beta}
Let $\overline{\cL}/K$ be a Galois extension, let $L/K$ run through an increasing family $\cF$ of finite Galois subextensions of $\overline{\cL}/K$ and for each $L$, let $\cL_L/L$ be a subextension of $\overline{\cL}/L$, such that $\cL_L/K$ is infinite Galois and $\bigcup_L \cL_L = \overline{\cL}$. Then

\[ \lim_{L \in \cF} \frac{1}{[L:K]}\beta(\cL_L/L) = \beta(\overline{\cL}/K). \]
\end{lm}

\begin{proof}
We have:

\begin{eqnarray*}
\lim_L \frac{1}{[L:K]} \beta(\cL_L/L) &=& \lim\limits_L \frac{1}{[L:K]} \sum\limits_{\fp_L \in \Sigma_L^f} \frac{1}{e_{\fp_L}(\cL_L/L)f_{\fp_L}(\cL_L/L)} \log \frac{\N\fp_L^{f_{\fp_L}(\cL_L/L)}}{\N\fp^{f_{\fp_L}(\cL_L/L)} - 1} \\
&=& \lim\limits_L \frac{1}{[L:K]} \sum\limits_{\fp \in \Sigma_K^f} \sum\limits_{\fp_L \in S_{\fp}(L)} \frac{1}{e_{\fp_L}(\cL_L/L)f_{\fp_L}(\cL_L/L)} \log \frac{\N\fp_L^{f_{\fp_L}(\cL_L/L)}}{\N\fp^{f_{\fp_L}(\cL_L/L)} - 1},
\end{eqnarray*}

\noindent where $S_{\fp}(L)$ denotes the set of all primes of $L$ lying over $\fp$. Choosing now one representative $\fp_L \in S_{\fp}(L)$ (note that as $\cL_L/K$ is assumed to be Galois, all these representatives have the same properties in the tower $\cL_L/L$) and noting that $\sharp S_{\fp}(L) = \frac{[L:K]}{e_{\fp}(L/K)f_{\fp}(L/K)}$, we see that the above expression is equal to

\begin{eqnarray*}
&=& \lim_L \frac{1}{[L:K]} \sum_{\fp \in \Sigma_K^f} \frac{[L:K]}{e_{\fp}(L/K)f_{\fp}(L/K)} \left( \frac{1}{e_{\fp_L}(\cL_L/L)f_{\fp_L}(\cL_L/L)} \log \frac{\N\fp_L^{f_{\fp_L}(\cL_L/L)}}{\N\fp_L^{f_{\fp_L}(\cL_L/L)} - 1} \right) \\
&=& \lim_L \sum_{\fp \in \Sigma_K^f} \frac{1}{e_{\fp}(\cL_L/K)f_{\fp}(\cL_L/K)} \log \frac{\N\fp^{f_{\fp}(\cL_L/K)}}{\N\fp^{f_{\fp}(\cL_L/L)} - 1} \\
&=& \beta(\overline{\cL}/K),
\end{eqnarray*}

\noindent where we used transitivity of $e_{\fp}$ and $f_{\fp}$ for the first and the absolute convergence of the involved series and $\bigcup_L \cL_L = \overline{\cL}$ for the last equation. Again, as in the proof of Proposition \ref{prop:elegant_comp_for_almost_tame_towers}, we are allowed to interchange the limit and the sum. \qedhere

\end{proof}

% \begin{eqnarray*} 
% \frac{\lambda_{\rel}(\cL_n^H/L)}{\mu_{\rel}(\cL_n^H/L)} - \frac{\lambda_{\rel}(\cL_n/L)}{\mu_{\rel}(\cL_n/L)} = \sum_{\fp \in \Sigma_L^f} \frac{1}{e_{\fp}(\cL_n^H/L)f_{\fp}(\cL_n^H/L)} \log \frac{\N\fp^{f_{\fp}(\cL_n^H/L)}}{\N\fp^{f_{\fp}(\cL_n^H/L)} - 1 } \\
% - \sum_{\fp \in \Sigma_L^f} \frac{1}{e_{\fp}(\cL_n/L)f_{\fp}(\cL_n/L)} \log \frac{\N\fp^{f_{\fp}(\cL_n/L)}}{\N\fp^{f_{\fp}(\cL_n/L)} - 1 }
% \end{eqnarray*}

\begin{comment}
Let $S$ be a finite set of primes of $K$ and let $K(\mu)_S/K = \bigcup_n K(\mu_n)_S$ be the union of all finite extensions of $K$ which can be written as a finite extension unramified outside $S(K(\mu_n))$ of a cyclotomic extension $K(\mu_n)$ of $K$. We denote by $\Gal_{K,S}(\mu)$ the Galois group of $K(\mu)_S/K$.

Let us consider the following condition on $K(\mu)_S/K$:

\begin{itemize}
\item[(*)] For any two primes $\bap,\baq$ of $K(\mu)_S$ the intersection of decomposition subgroups $D_{\bap} \cap D_{\baq}$ inside $\Gal_{K,S}(\mu)$ is trivial.
\end{itemize}
\end{comment}

%************************************************************************************************************************************************************
%************************************************************************************************************************************************************

\section{Characterization of decomposition subgroups} \label{sec:char_of_dec_groups}\mbox{}

Let $S$ be a finite set of primes of $K$ with $S_p \cup S_{\infty} \subseteq S$. We now consider the question of how to describe decomposition subgroups of primes outside $S$ using anabelian information.

\subsection{Some preliminaries}\label{sec:Conditions_C1_C3_etc} \mbox{}

For a pro-finite group $H$, let $H_p$ denote a pro-$p$-Sylow subgroup of it. Consider the following condition on $K,S$:

\begin{itemize}
\item[(C1)] for any two different primes $\bap, \bar{\fq}$ of $K_S$, the intersection $D_{\bap,p} \cap D_{\bar{\fq},p} \subseteq D_{\bap,p}$ is not open in $D_{\bap,p}$.
\end{itemize}

\noindent Here $D_{\bap,p}$ denotes a pro-$p$-Sylow subgroup of the decomposition subgroup $D_{\bap}$ of $\bap$ in $K_S/K$. Note that (C1) for $K,S$ implies (C1) for $L,S$, where $K_S/L/K$ is a finite subextension. Note also that due to $S_p \cup S_{\infty} \subseteq S$, the decomposition subgroup $D_{\bap} \subseteq \Gal_{K,S}$ at a prime $\bap \not\in S(K_S)$ is pro-cyclic with infinite pro-$p$-Sylow subgroup. We have the following standard lemma.

\begin{lm}\label{lm:FKS_lemma}
Assume (C1) holds for $K,S$. 
\begin{itemize}
\item[(i)] If $x \in \Gal_{K,S}$ normalizes an open subgroup of $D_{\bap,p}$, then $x \in D_{\bap}$.
\item[(ii)] If $Z \subseteq \Gal_{K,S}$ is an abelian subgroup with infinite cyclic pro-$p$-Sylow subgroup, then there is at most one prime $\bap$ of $K_S$ (which then has to lie outside $S(K_S)$), such that $Z \cap D_{\bap} \subseteq D_{\bap}$ is open. Moreover, in that case, we have $Z \subseteq D_{\bap}$.
\end{itemize}
\end{lm}

\begin{proof}
(i):  We have $x D_{\bap,p} x^{-1} = D_{x\bap,p}$. Let $U \subseteq D_{\bap,p}$ be an open subgroup normalized by $x$. Thus $D_{\bap,p} \supseteq U = xUx^{-1} \subseteq D_{x\bap,p}$, i.e., the open subgroup $U$ of $D_{\bap,p}$ is contained in the intersection of $D_{\bap,p}$ and $D_{x\bap,p}$. By (C1) we get $x \bap = \bap$, i.e., $x \in D_{\bap}$. \\

(ii): Let $\bap$ be a prime, such that $Z \cap D_{\bap} \subseteq D_{\bap}$ is open. In particular, $Z \cap D_{\bap,p} \subseteq D_{\bap,p}$ is open. As $Z$ is abelian, $Z$ normalizes $Z \cap D_{\bap,p}$, hence $Z \subseteq D_{\bap}$ by part (i) of the lemma. Note that $\bap$ must lie outside $S$, as for $\bap \in S(K_S)$, the $p$-Sylow subgroup $D_{\bap,p} \cong \bZ_p \ltimes \bZ_p$ do not contain an open abelian subgroup. If $Z \cap D_{\bap_i} \subseteq D_{\bap_i}$ is open for $i = 1,2$, then $Z \subseteq D_{\bap_i}$ and hence $Z_p \subseteq D_{\bap_i,p}$ (as $D_{\bap_i}$ is abelian). Thus $\bZ_p \cong Z_p \subseteq D_{\bap_1,p} \cap D_{\bap_2,p}$. As $D_{\bap_i, p} \cong \bZ_p$, we see that $D_{\bap_1,p} \cap D_{\bap_2,p} \subseteq D_{\bap_1,p}$ must be open. By condition (C1), $\bap_1 = \bap_2$. \qedhere
% If there were two primes $\bap_1,\bap_2$ with this property, then $\bZ_p \cong D_{\bap_i,p} \supseteq D_{\bap_i,p} \cap Z$ is open, hence $D_{\bap_i,p} \cap Z \cong \bZ_p$. As $Z_p \cong \bZ_p$, we see that $D_{\bap_1,p} \cap D_{\bap_2,p} \cap Z_p \subseteq Z_p \cong \bZ_p$ must be open, hence also $D_{\bap_1,p} \cap D_{\bap_2,p} \subseteq D_{\bap_1}$ is open (both isomorphic to $\bZ_p$). This implies $\bap_1 = \bap_2$.
\end{proof}

Unfortunately, condition (C1) is not enough to apply our method to $K_S/K$ with $S$ finite, and we have to assume either one of the following stronger conditions additionally to (C1). For a subgroup $Z$ of a pro-finite group $U$, we denote by $\langle\langle Z \rangle\rangle_U$ the smallest normal closed subgroup of $U$ containing $Z$.

\begin{itemize}
\item[(C2)] Condition (C1) holds and if $Z \subseteq \Gal_{K,S}$ is a closed pro-cyclic subgroup with infinite pro-$p$-Sylow subgroup, then the following holds:
    \begin{itemize}
	\item[$\bullet$] If $Z \subseteq D_{\bap}$ is open for a prime $\bap$ of $K_S$ lying over $\fp$ (necessarily $\fp \not\in S$), then
		  \[ \lim_{K_S^Z/L/K} \beta(K_S^{Z_L}/L) = \log \frac{\N\fp^{f_{\bap}(K_S^Z/K)}}{\N\fp^{f_{\bap}(K_S^Z/K)} - 1}, \] % \frac{1}{f_{\fp}(K_S^Z/K)} 
	    \noindent where $Z_L := \langle\langle Z \rangle\rangle_{\Gal_{K_S/L}}$ and $f_{\bap}(K_S^Z/K) = (D_{\bap} : Z)$ is the index of $\bap$ in $K_S^Z/K$.
	\item[$\bullet$] If $Z \cap D_{\bap}$ is not open in $D_{\bap}$ for any $\bap$, then
		  \[ \lim_{K_S^Z/L/K} \beta(K_S^{Z_L}/L) = 0. \] 
    \end{itemize}
\item[(C3)] For any finite subextension $K_S/L/K$, for any two primes $\fp \neq \fq$ of $L$, the decomposition subgroup of $\fq$ in the extension $K_S^{\{\fp\}}/L$ is infinite.
\end{itemize}

\begin{lm}\label{lm:C3_conn_to_C2}
Condition (C3) for $K,S$ implies (C1) and the first part of (C2), i.e., for any closed pro-cyclic subgroup $Z \subseteq \Gal_{K,S}$ with infinite pro-$p$-Sylow subgroup, if $Z \subseteq D_{\bap}$ is open for a prime $\bap$ of $K_S$, then
\[ \lim_{K_S^Z/L/K} \beta(K_S^{Z_L}/L) = \log \frac{\N\fp^{f_{\bap}(K_S^Z/K)}}{\N\fp^{f_{\bap}(K_S^Z/K)} - 1}. \]
\end{lm}

\begin{proof} 
% If $\bap, \bar{\fq} \in S^f(K_S)$, then $D_{\bap,p} \cap D_{\bar{\fq},p} \subseteq D_{\bap,p}$ is not open by \cite{Iv} Proposition 2.4. If $\bap \in S(K_S)$ and $\bar{\fq} \not\in S(K_S)$, then $D_{\bap,p} \cap D_{\bar{\fq},p} \subseteq D_{\bap,p}$ can not be open as $D_{\bap,p} \cong \bZ_p \ltimes \bZ_p$ and $D_{\bar{\fq},p} \cong \bZ_p$. If $\bap \not\in S(K_S)$ and $\bar{\fq}$ arbitrary, then $D_{\bap,p} \cap D_{\bar{\fq},p} \subseteq D_{\bap,p}$ not open follows from (C3). 
(C1) easily follows from (C3). Assume now, $Z \subseteq D_{\bap}$ is open for some $\bap$ with restriction $\fp$ to $K$. Then necessarily $\bap \not\in S(K_S)$. Let $Z \subseteq U \subseteq \Gal_{K,S}$ be open and write $L := K_S^U$ and $Z_L := \langle \langle Z \rangle \rangle_U$.  Let $\fp_L$ denote the restriction of $\bap$ to $L$. Then $Z_L \subseteq \langle \langle D_{\bap/\fp_L} \rangle \rangle_U$ and by (C3) we see that any prime $\fq \neq \fp_L$ of $L$ has infinite degree in the Galois tower $K_S^{\langle \langle D_{\bap/\fp_L} \rangle \rangle_U}/L$ and hence also in the Galois tower $K_S^{Z_L}/L$. Thus 

\[ \beta(K_S^{Z_L}/L) =  \frac{1}{f_{\fp_L}(K_S^{Z_L}/L)} \log \frac{\N\fp_L^{f_{\fp_L}(K_S^{Z_L}/L)}}{\N\fp_L^{f_{\fp_L}(K_S^{Z_L}/L)} - 1}. \]

\noindent For $U \supseteq Z$ small enough, we have $U \cap D_{\bap/\fp} = Z$, hence $f_{\bap}(L/K) = (D_{\bap/\fp}:Z) = f_{\bap}(K_S^Z/K)$ and $f_{\fp_L}(K_S^{Z_L}/L) = e_{\fp_L}(K_S^{Z_L}/L) = 1$. Thus $\N\fp_L^{f_{\fp_L}(K_S^{Z_L}/L)} = (\N\fp^{f_{\bap}(L/K)})^{f_{\fp_L}(K_S^{Z_L}/L)} = \N\fp^{f_{\bap}(K_S^Z/K)}$ and the lemma follows. \qedhere
\end{proof}

\begin{Def}\label{Def:lie_under} Let $U \subseteq \Gal_{K,S}$ be an open subgroup and $L := K_S^U$. Let $H \triangleleft U$ be a closed normal subgroup, which is not open. We say that a prime $\fp$ of $L$ \textbf{lies under} $H$, if for some (any) extension $\bap$ of $\fp$ to $K_S$, the inclusion $D_{\bap/\fp} \cap H \subseteq D_{\bap/\fp}$ is open (equivalently, $[\fp_H \colon \fp] < \infty$, where $\fp_H := \bap|_{K_S^H}$). For a set $R \subseteq \Sigma_L^f$, we define:
\[ z_R(U;H) := \text{the number of primes in $L$ lying outside $R$ and under $H$.} \]
\end{Def}

\begin{Prop}[Criterion 1]\label{prop:crit1} Let $H \triangleleft U$, $L$ be as in Definition \ref{Def:lie_under}. Assume that either $L/\bQ$ is almost normal, or that generalized Riemann hypothesis holds for all involved number fields. Assume that (C3) holds for $K,S$. We have:
\begin{itemize}
\item[(i)]  $z_S(U;H) = 0 \LRar \beta(K_S^H/L) - \sum_{\fp \in S^f(L)} \frac{1}{f_{\fp}e_{\fp}} \log \frac{\N\fp^{f_{\fp}}}{\N\fp^{f_{\fp}} - 1} = 0$ (where $f_{\fp} := f_{\fp}(K_S^H/L)$, etc.). 
\item[(ii)] $z_S(U;H) > 1 \LRar \exists H_1, H_2 \subseteq H$ which are closed and normal in $U$ such that $z_S(U;H_i) > 0$ for $i = 1,2$ and $z_S(U;H_1 \cap H_2) = 0$.
\end{itemize}
\end{Prop}

\begin{proof}
(i) follows directly from the definition of $\beta$. (ii):  Assume first $z_S(U;H) > 1$. Then there are two different primes $\fp_1,\fp_2$ of $L$, both lying outside $S$ and under $H$. For $i=1,2$ let $\bap_i$ be an extension of $\fp_i$ to $K_S$ and let $H_i := \langle\langle D_{\bap_i/\fp_i} \cap H \rangle\rangle_U$ be the closed normal subgroup of $U$ generated by $D_{\bap_i/\fp_i} \cap H$. Then obviously, $\fp_i$ lies under $H_i$, hence $z_S(U;H_i) > 0$. Further, for any prime $\fp \neq \fp_1$ of $L$ with some extension $\bap$ to $K_S$, we have:

\[ D_{\bap/\fp} \cap H_i \subseteq D_{\bap/\fp} \cap \langle\langle D_{\bap_1/\fp_1} \rangle\rangle_U \subseteq D_{\bap/\fp}, \]

\noindent and the second inclusion is not open by (C3). Hence (doing the same for $\fp_2$ instead of $\fp_1$) no prime lies under $H_1 \cap H_2$. In particular, $z_S(U; H_1 \cap H_2) = 0$.

Conversely, assume there are two closed subgroups $H_1,H_2 \subseteq H$, normal in $U$, such that $z_S(U;H_i) > 0$ for $i = 1,2$ and $z_S(U;H_1 \cap H_2) = 0$. Let $\fp_i$ be a prime of $L$ lying outside $S$ and under $H_i$. Then $\fp_i$ also lies under $H$. If we would have $\fp_1 = \fp_2 =: \fp$, then $D_{\bap/\fp} \cap H_i \subseteq D_{\bap/\fp}$ would be an open inclusion for $i = 1,2$. But then also $D_{\bap/\fp} \cap H_1 \cap H_2 \subseteq D_{\bap/\fp}$ would be open, i.e., $\fp$ would lie under $H_1 \cap H_2$, which contradicts $z_S(U; H_1 \cap H_2) = 0$. Thus $\fp_1 \neq \fp_2$, and hence $z_S(U; H) > 1$. 
\end{proof}

%************************************************************************************************************************************************************
%************************************************************************************************************************************************************

\subsection{Subgroups of decomposition behavior} \mbox{}

We now introduce candidates from decomposition subgroups. The definitions depend on which condition we pose on $K_S/K$.

\begin{Def}\label{Def:dec_behavior_C2}
Assume that (C2) holds for $K_S/K$. Let $Z \subseteq \Gal_{K,S}$ be a closed pro-cyclic subgroup with infinite pro-$p$-Sylow subgroup. We say that $Z$ has \textbf{(C2)-decomposition behavior} if there is a constant $C > 0$ such that
\[  \lim_{K_S^Z/L/K} \beta(K_S^{Z_L}/L) = C, \]
\noindent where $Z_L := \langle \langle Z \rangle\rangle_{\Gal_{L,S}}$.
\end{Def}

\begin{Def}\label{Def:dec_behavior_C3} Assume that (C3) holds for $K_S/K$. Let $Z \subseteq \Gal_{K,S}$ be a closed pro-cyclic subgroup with infinite pro-$p$-Sylow subgroup. We say that $Z$ has \textbf{(C3)-decomposition behavior} if the following conditions are satisfied for all open subgroups $U \subseteq \Gal_{K,S}$ containing $Z$:
\begin{itemize}
% \item[(1)] $\langle\langle Z \rangle\rangle_U \subseteq U$ is not open
\item[(1)] $z_S(U; \langle\langle Z \rangle\rangle_U) = 1$
\item[(2)] there is a constant $C > 0$ and an open subgroup $Z \subset U_0 \subseteq \Gal_{K,S}$, such that if $Z \subset U \subseteq U_0$, then 
\[ \beta(K_S^{Z_L}/L) = C, \]
\noindent where $L := K_S^U$ and $Z_L := \langle\langle Z \rangle\rangle_U$. 
\end{itemize}
\end{Def}

\begin{rem}
From Lemma \ref{lm:C3_conn_to_C2} and the proof of Theorem \ref{thm:crit_2} one gets the impression that (C2) is the 'better' condition. Nevertheless, we do not know how to deduce the second part of (C2) from (C3) and therefore we work with both conditions separately.
\end{rem}

\begin{thm}\label{thm:crit_2}
Let $S \supseteq S_p \cup S_{\infty}$ be a finite set of primes of $K$. Assume that (C2) (resp. (C3)) holds for $K,S$. Assume that the generalized Riemannian hypothesis holds for all involved number fields. Let $Z \subseteq \Gal_{K,S}$ be a closed pro-cyclic subgroup with infinite pro-$p$-Sylow subgroup. Then $Z$ has (C2)- (resp. (C3)-) decomposition behavior if and only if there is a prime $\bap$ of $K_S$, such that $Z \cap D_{\bap} \subseteq D_{\bap}$ is open. In this case $Z \subseteq D_{\bap}$, the prime $\bap$ is unique and $\bap \not \in S(K_S)$. Moreover, the norm of the restriction $\fp$ of $\bap$ to $K$ is equal to 

\[ \N\fp = \frac{e^C}{e^C - 1}, \]

\noindent where $C > 0$ is the constant attached to $D_{\bap/\fp}$ in Definition \ref{Def:dec_behavior_C2} (resp. \ref{Def:dec_behavior_C3}).
\end{thm}

\begin{proof}
Assume first (C2) holds for $K,S$. By Lemma \ref{lm:FKS_lemma}, the following two statements about $Z$ are equivalent: 'there is a prime $\bap$ of $K_S$ such that $Z \cap D_{\bap} \subseteq D_{\bap}$ is open' and 'there is a prime $\bap$ of $K_S$ such that $Z \subseteq D_{\bap}$ is open'. Thus in this case, the first statement of the theorem follows directly from (C2).

Now assume (C3) holds for $K,S$. Assume first there is a prime $\bap$ of $K_S$ with restriction $\fp$ to $K$ such that $Z \cap D_{\bap/\fp} \subseteq D_{\bap/\fp}$ is open.  For any $Z \subseteq U \subseteq \Gal_{K,S}$ open with fixed field $L$, put $\fp_L := \bap|_L$ and write $Z_L := \langle \langle Z \rangle \rangle_U$. Then $\fp_L$ lies under $Z_L$ and not in $S(L)$. If a further prime $\fq \neq \fp_L$ of $L$ with $\fq \not\in S$ would lie under $Z_L$, then the composition $D_{\baq/\fq} \cap Z_L \subseteq D_{\baq/\fq} \cap \langle \langle D_{\bap/\fp_L} \rangle \rangle_U \subseteq D_{\baq/\fq}$ would be open. Hence also the second inclusion would be open, and this contradicts $(C3)$. This shows part (1) in Definition \ref{Def:dec_behavior_C3}. As in the proof of Lemma \ref{lm:C3_conn_to_C2}, we see that $\beta(K_S^{Z_L}/L)$ gets constant (and $> 0$) if $U$ is small enough, more precisely, if $U \cap D_{\bap/\fp} = Z$. This shows that $Z$ has (C3)-decomposition behavior. 

Conversely, assume $Z$ has (C3)-decomposition behavior. Let $Z \subseteq U \subseteq \Gal_{K,S}$ be an open subgroup with fixed field $L$. By assumptions, there is a unique prime $\fp_L$ of $L$ lying outside $S$ and under $Z_L$. This uniqueness implies that if $Z \subseteq U^{\prime} \subseteq U \subseteq \Gal_{K,S}^{\bullet}$ with fixed fields $L^{\prime}, L$, then $\fp_{L^{\prime}}|_L = \fp_L$. Thus the sets $\{\fp_L\}$ with one element form a projective system, which limit is again a one element set, i.e., we obtain a unique prime $\bap$ of $K_S$ (lying outside $S$) with $\bap|_L = \fp_L$ for each $L$. 

Let us compute the numbers $\beta(K_S^{Z_L}/L)$ for any open $Z \subseteq U \subseteq \Gal_{K,S}$. By part (1) of Definition \ref{Def:dec_behavior_C3}, the prime $\fp_L$ is the only one lying under $Z_L$, i.e.,  $\fp_L$ is the only prime having finite norm in $K_S^{Z_L}/L$. Set $f_U := f_{\fp_L}(K_S^{Z_L}/L)$ and $t := \N\fp_L$. By definition of $\beta$ we compute:

\[ \beta(K_S^{Z_L}/L) = \frac{1}{f_U} \log \frac{t^{f_U}}{t^{f_U} - 1}. \]

\noindent Fix now two open subgroups $Z \subseteq U^{\prime} \subseteq U \subseteq U_0$ with fixed fields $L^{\prime},L$, where $U_0$ is as in part (2) of Definition \ref{Def:dec_behavior_C3}. Set $c := [\fp_{L^{\prime}}: \fp_L] = (D_{\bap/\fp_L} : D_{\bap/\fp_{L^{\prime}}})$. Then $\N\fp_{L^{\prime}} = \N\fp^c = t^c$ and by part (2) of Definition \ref{Def:dec_behavior_C3} we must have:

\[ \frac{1}{f_{U^{\prime}}} \log \frac{t^{cf_{U^{\prime}}}}{t^{cf_{U^{\prime}}} - 1} = \beta(K_S^{Z_{L^{\prime}}}/L^{\prime}) = \beta(K_S^{Z_L}/L) =  \frac{1}{f_U} \log \frac{t^{f_U}}{t^{f_U} - 1}, \]

\noindent or, equivalently,

\[ (c - 1)\log t = \log \left( \frac{ (t^{c f_U^{\prime} }- 1)^{1/f_{U^{\prime}}}  }{  (t^{f_U} - 1)^{1/f_U}  }\right). \]

\noindent Applying $\exp$ and taking $f_U f_{U^{\prime}}$-th power we obtain
\[ t^{(c - 1)f_U f_{U^{\prime}}}(t^{f_U} - 1)^{f_{U^{\prime}}} = (t^{c f_{U^{\prime}}} - 1)^{f_U}. \]

\noindent All numbers in this equation are positive integers and if $c \neq 1$, then the left side would be divisible by $t$ (which is a prime power $> 1$), whereas the right side would not. This contradiction shows $c = 1$ and therefore, the intersection $D_{\bap/\fp} \cap U$ is independent of the open subgroup $Z \subseteq U \subseteq U_0$. Hence 
\[ D_{\bap/\fp} \cap Z = D_{\bap/\fp} \cap \bigcap_{Z \subseteq U} U = D_{\bap/\fp} \cap U_0 \]
\noindent is open in $D_{\bap/\fp_K}$. This finishes the proof of the first statement of the theorem.

Uniqueness of $\bap$ follows from Lemma \ref{lm:FKS_lemma} as (C2), (C3) imply (C1). The $p$-Sylow subgroup of a decomposition subgroup of a prime in $S$ is isomorphic to $\bZ_p \ltimes \bZ_p$, hence do not contain any open abelian subgroup. Thus $\bap \not\in S(K_S)$. Further, $Z \subseteq D_{\bap}$ follows from Lemma \ref{lm:FKS_lemma}. This shows the second statement of the theorem. The last statement follows directly from (C2) resp. is an elementary computation. \qedhere
\end{proof}

\begin{cor} \label{cor:determining_the_dec_groups}
With the assumptions as in Theorem \ref{thm:crit_2}, the decomposition subgroups inside $\Gal_{K,S}$ of primes lying outside $S$ are exactly the maximal pro-cyclic subgroups of $\Gal_{K,S}$ with infinite pro-$p$-Sylow subgroup of (C2)- (resp. (C3)-) decomposition behavior.
\end{cor}

\begin{proof}
The $\bZ_p$-cyclotomic extension of $K$ is contained in $K_S$. Thus the decomposition subgroup inside $\Gal_{K,S}$ of any prime $\fp \not\in S$ is pro-cyclic with infinite pro-$p$-Sylow subgroup. Now the corollary follows from Theorem \ref{thm:crit_2} and Lemma \ref{lm:FKS_lemma}.
\end{proof}

%************************************************************************************************************************************************************
%************************************************************************************************************************************************************

\section{Anabelian geometry} \label{sec:anab_geom}

The goal of this section is to apply the theory developed in Sections \ref{sec:some_invs_of_infinite_towers}, \ref{sec:char_of_dec_groups} to anabelian geometry and to prove Theorem \ref{thm:main_thm}. Let $K_S^{\tame}$ denote the maximal tamely ramified subextension of $K_S/K$. First we need a further condition:

\begin{itemize}
\item[(C4)] there is a finite subextension $K_S/M/K$, such that for each prime $\fp \not\in S(M)$, there is no $M_S^{\tame, \{\fp\}}/L/M$ finite, such that $M_S^{\{\fp\},\tame}/L$ is unramified (in particular, the extension $M_S^{\{\fp\},\tame}/M$ is infinite) and for any closed normal subgroup $H$ of $\Gal_{M_S/M}$, which is generated by conjugates of one element, $\bigcup_{M_S^H/L/M} L_S^{\tame}$ has no prime of finite norm.
\end{itemize}

\begin{rems} Condition (C4) is in contrast to (C2),(C3) purely of technical nature. We are convinced that it can be removed or at least weakened. Here are two remarks concerning it:
\begin{itemize}
\item[(i)] If $M$ satisfies the requirements posed in (C4), then also any finite $K_S/M^{\prime}/M$ does. If there would be some finite subextension $M$ such that $M_S^{\tame}$ has no prime of finite norm, then everything would be much easier. But, we need not to assume this and indeed, the last condition required in (C4) is much weaker: if $H$ is a normal subgroup of $\Gal_{M_S/M}$, generated by conjugates of one element, then the extension $M_S^H/M$ is very big at least if $M$ is totally imaginary (indeed, $M_S^H/M$ contains at least $r_2(M)$ independent $\bZ_p$-extensions of $M$).
\item[(ii)] Technically, we need (C4) for the crucial application of Proposition \ref{prop:elegant_comp_for_almost_tame_towers}, which we are finally forced to apply to extensions $\cL_n = L_S^{\tame}$ (constant for all $n$), as there are no other towers with reasonable properties. It would be much nicer, if one would be able to apply Proposition \ref{prop:elegant_comp_for_almost_tame_towers} to, say, $\cL_n^{\prime} := L(\mu_{p^n})_S^{\tame}$: then the last part of (C4) gets obsolete. Unfortunately, we are not able to obtain full control over $\mu_{\rel}(\cL_n^{\prime}/L)$, as wild ramification is in the game (although this wild ramification 'just' comes from the cyclotomic $\bZ_p$-tower).
\end{itemize}
\end{rems}

\begin{proof}[Proof of Theorem \ref{thm:main_thm}]
For each $K_{1,S_1}/L_1/K_1$, let $L_2$ be the field corresponding to $L_1$ via $\sigma$. By assumption (1) in Theorem \ref{thm:main_thm} we have a local correspondence at the boundary, i.e., there is a $\Gal_{K_1,S_1}$-invariant map

\[ \sigma_{\ast} \colon S_{1,f}(K_{1,S_1}) \stackrel{\sim}{\rar} S_{2,f}(K_{2,S_2}), \]

\noindent characterized by $\sigma(D_{\bap}) = D_{\sigma_{\ast}(\bap)}$, which induces bijections $\sigma_{\ast,L_1}$ at each finite level $K_{1,S_1}/L_1/K_1$. Moreover, the decomposition subgroups of primes in $S_i$ are full local groups, hence $\sigma_{\ast,L_1}$ preserves residue characteristics, the absolute degrees of primes, the inertia degrees and the ramification indices.

Further, if $I_{\bap}$ denote the inertia and $R_{\bap}$ the wild inertia subgroup of $D_{\bap}$, then $\sigma(I_{\bap}) = I_{\sigma_{\ast}(\bap)}$ and $\sigma(R_{\bap}) = R_{\sigma_{\ast}(\bap)}$. In particular, if $L_1, L_2$ correspond via $\sigma$, then also $L_{1,S_1}^{\tame}$, $L_{2,S_2}^{\tame}$ do, and we have $g_{M_1/L_1} = g_{M_2/L_2}$ if $M_1/L_1$ is tame. Moreover, this implies, that if $K_{1,S_1}/\cL_1/L_1/K_1$ is a tamely ramified subextension, then 

\begin{equation}\label{eq:tametower_equal_mu_rels}
\mu_{\rel}(\cL_1/L_1) = \mu_{\rel}(\cL_2/L_2). 
\end{equation}

% $\chi_{K_1,p} = \sigma \circ \chi_{K_2,p}$ and that
Due to Lemma \ref{lm:FKS_lemma} it is enough to show the existence of $\sigma_{\ast}$ after a finite extension of $K$ inside $K_S$. Thus by \cite{Iv} Theorem 1.1, we can assume that $\mu_p \subseteq K$. Moreover, by condition (C4), we can assume that for any finite subextension $K_S/M/K$ and any $\fp \in \Sigma_M^f$, the extension $M_S^{\{\fp\},\tame}/M$ is infinite, there is no $M_S^{\tame, \{\fp\}}/L/M$ finite, such that $M_S^{\{\fp\},\tame}/L$ is unramified and that $\bigcup_{K_S^H/L/K} L_S^{\tame}$ has no prime of finite norm for any $H$ as in (C4). 

Due to Corollary \ref{cor:determining_the_dec_groups}, to show the existence of $\sigma_{\ast}$, it is enough to show that for any pro-cyclic subgroup $Z_1 \subseteq \Gal_{K_1,S_1}$ with infinite pro-$p$-Sylow subgroup, the following holds: if $Z_1$ has (C2)- (resp. (C3)-) decomposition behavior, then $Z_2 := \sigma(Z_1)$ also has. Due to (C4), we can restrict attention only to such subgroups $Z_1$, for which $K_{1,S_1}^{\langle\langle Z_1 \rangle\rangle_{\Gal_{K_1,S_1}}, \tame}/K_1$ is infinitely ramified (and hence, in particular, infinite). Note that then the same is true for $Z_2$ instead of $Z_1$. By definition of decomposition behavior, it is enough to show that for all $K_{1,S_1}^{Z_1}/L_1/K_1$ finite, one has 

\[ \beta(L_{1,S_1}^{ Z_{1,L_1} }/L_1) = \beta(L_{2,S_2}^{ Z_{2,L_2} }/L_2), \] 

\noindent where $Z_{i,L_i}$ is the normal subgroup of $\Gal_{L_1,S_1}$ generated by $Z_i$. By the last statement of Proposition \ref{prop:elegant_comp_for_almost_tame_towers} (applied to $L := M_i$ and $\cL_{\infty} := \cL_n := M_{i,S_i}^{\tame}$), we know that (we omit the index $i = 1,2$):

\begin{equation} \label{eq:betas_which_should_coincide}
\beta(L_S^{ Z_L }/L) = \lim_{L_S^{Z_L}/M/L} \frac{1}{[M:L]} (\beta(M_S^{\tame} \cap L^{Z_L}/M ) - \beta(M_S^{\tame}/M) ) 
\end{equation}

\noindent Note that the assumptions of Proposition \ref{prop:elegant_comp_for_almost_tame_towers} are satisfied. Indeed,

\begin{itemize}
 \item[(i)] we know that $K_S^{Z_K} \cap K_S^{\tame}/K$ is infinitely ramified, and $L_S^{Z_L}  \cap M_S^{\tame} \supseteq K_S^{Z_K} \cap K_S^{\tame}$. Hence also $M_S^{\tame} \cap L^{Z_L}/M$ is infinite and ramified, i.e., $\mu_{\rel}(M_S^{\tame}/M) \leq \mu_{\rel}(M_S^{\tame} \cap L^{Z_L}/M) < \infty$,
 \item[(ii)] on the other side $0 < \mu_{\rel}(M_S^{\tame}/M) \leq \mu_{\rel}(M_S^{\tame} \cap L^{Z_L}/M)$ by Lemma \ref{lm:mu_is_strictly_pos_in_tame_towers}.
 \item[(iii)] $\bigcup_{L_S^{Z_L}/M/L} M_S^{\tame} \supseteq \bigcup_{K_S^{Z_K}/M/K} M_S^{\tame}$ has no prime of finite norm by (C4).
\end{itemize}

\noindent Finally, by the first statement of Proposition \ref{prop:elegant_comp_for_almost_tame_towers}, the right side of \eqref{eq:betas_which_should_coincide} can be expressed in terms of $\lambda_{\rel}$ and $\mu_{\rel}$ of the tamely ramified towers $M_S^{\tame} \cap L^{Z_L}/M$ and $M_S^{\tame}/M$. Hence the right sides of \eqref{eq:betas_which_should_coincide} coincide for $i=1,2$. Thus $Z_1$ has decomposition behavior if and only if $Z_2$ has and we are done. This all shows the existence of the compatible bijections $\sigma_{\ast,L_1}$. Moreover, it is clear from Theorem \ref{thm:crit_2} and preceeding computations, that the maps $\sigma_{\ast,L_1}$ preserve norms of primes.
\end{proof}

\section{The $p$-adic volume of the unit lattice}\label{sec:padic_regulator}

The results of this section are independent from the rest of the article. Let $p$ be an odd prime, $K$ a number field, $S \supseteq S_{\infty} \cup S_p$ a finite set of primes of $K$. In contrast to the usual regulator, it is possible to reconstruct the $p$-adic volume $\vol_p(K)$ of the unit lattice of $K$ from the fundamental group $\Gal_{K,S}$ (+ some more information). For a definition of $\vol_p(K)$ we refer to \cite{NSW} 10.3.3. 

\begin{prop} Let $K$ be a number field, $S \supseteq S_{\infty}$ a finite set of primes of $K$, such that at least two rational primes lie in $\caO_{K,S}^{\ast}$ and let $p$ be one of them. Assume $p > 2$ or $K$ totally imaginary. Assume the Leopoldt conjecture holds for $K$ and $p$. Then 
\[ (\Gal_{K,S},p,\chi_p) \rightsquigarrow \vol_p(K). \]
\end{prop}

\begin{proof}
By \cite{Iv} Theorem 1.1, the given information is enough to reconstruct the position of decomposition subgroups of primes in $S^f$ inside $\Gal_{K,S}$ and hence also in its pro-$p$-quotient $\Gal_{K,S}^{(p)}$. For a prime $\bap \in (S^f \sm S_p)(K_S^{(p)})$, we have a map with open image $D_{\bap,K_S^{(p)}/K} \rar \bZ_p$, which is induced by the cyclotomic character $\chi_p^{\prime} \colon \Gal_{K,S}^{(p)} \rar \bZ_p$. The inertia subgroups $I_{\bap,K_S^{(p)}/K} \subseteq D_{\bap,K_S^{(p)}/K}$ are the kernels of this map. Therefore we can reconstruct the quotient $\Gal_{K,S_p}^{(p)}$ of $\Gal_{K,S}^{(p)}$ together with the decomposition subgroups at $S_p$ by dividing out the normal subgroup generated by $I_{\bap,K_S^{(p)}/K}$ for $\bap \in S^f \sm S_p$. Hence we can also reconstruct the following exact sequence from class field theory (cf. \cite{NSW} 8.3.21):

\[ 0 \rar \overline{\caO_{K,S_p}^{\ast}} \rar \prod_{\fp \in S_p} K_{\fp}^{\ast,(p)} \rar \Gal_{K,S_p}^{\ab,(p)} \rar \Cl_{S_p}(K)^{(p)} \rar 0 \]

\noindent where the upper index $(p)$ denotes the pro-$p$ completion and $\overline{\caO_{K,S_p}^{\ast}}$ denotes the closure of the image of $\caO_{K,S_p}^{\ast}$ in $\prod_{\fp \in S_p} K_{\fp}^{\ast,(p)}$. Note that the decomposition groups $D_{\bap,K_S/K}$ are the full local groups by \cite{CC} and hence by class fields theory $K_{\fp}^{\ast,(p)} \cong \Gal_{K_{\fp}}^{\ab,(p)} \cong D_{\bap,K_S/K}^{\ab,(p)}$. Let $U_{\fp} \subset K_{\fp}^{\ast}$ be the units of $\caO_{K_{\fp}}$. The Leopoldt conjecture and \cite{NSW} 10.3.13 with $S = S_p \cup S_{\infty}$, $T = \emptyset$ shows the exactness of the rows of the diagram

\[
\begin{xy} 
\xymatrix{ 
0 \ar[r] & \overline{\caO_K^{\ast}} \ar[r] \ar@{^{(}->}[d] & \prod_{\fp \in S_p} U_{\fp}^{(p)} \ar@{^{(}->}[d] \ar[r] & \Gal_{K,S_p}^{\ab,(p)} \ar@{=}[d] \ar[r] & \Cl_K^{(p)} \ar@{->>}[d] \ar[r] & 0 \\
0 \ar[r] & \overline{\caO_{K,S_p}^{\ast}} \ar@{^{(}->}[r] & \prod_{\fp \in S_p} K_{\fp}^{\ast,(p)} \ar[r] & \Gal_{K,S_p}^{\ab,(p)} \ar[r] & \Cl_{K,S_p}^{(p)} \ar[r] & 0 \\
}
\end{xy}
\]

\noindent where $\overline{\caO_K^{\ast}}$ is the closure of the image of $\caO_K^{\ast}$ inside $\prod_{\fp \in S_p} U_{\fp}^{(p)}$. We can reconstruct the second vertical map from the given data ($U_{\fp}^{(p)}$ correspond to the inertia subgroup via the reciprocity isomorphism $K_{\fp}^{\ast,(p)} \rar \Gal_{K_{\fp}}^{\ab,(p)}$). An easy diagram chase shows $\overline{\caO_K^{\ast}} = \overline{\caO_{K,S_p}^{\ast}} \cap \prod_{\fp \in S_p} U_{\fp}^{(p)}$. Hence we can reconstruct the upper left horizontal map in the above diagram. Apply $- \otimes_{\bZ_p} \bQ_p/\bZ_p$ to it; now the proposition follows from \cite{NSW} 10.3.8 as $r_1(K),r_2(K)$ can be reconstructed from the given data and as the natural map $\caO_K^{\ast} \otimes_{\bZ} \bZ_p \tar \overline{\caO_K^{\ast}}$ is an isomorphism by Leopoldt's conjecture.
\end{proof}

\subsection*{Acknowledgments}
First of all I want to mention, that the idea to search for a group theoretic criterion for small norm of primes in number fields was given to me originally by Jakob Stix during writing of my Ph.D. thesis. I want to thank Jakob Stix, Florian Pop, Jochen G\"artner and Johannes Schmidt for fruitful discussions on this subject, from which I learned a lot. Last but not least, I want to thank Malte Witte for helping me with some questions concerning Iwasawa theory. The author wants to thank the Technical Univeristy Munich, where this work was done.

%************************************************************************************************************************************************************
%************************************************************************************************************************************************************

\renewcommand{\refname}{References}

\end{document}